\documentclass[10pt]{article}
\UseRawInputEncoding
\usepackage{amsthm}
\usepackage{amsmath}
\usepackage{amssymb}
\usepackage{makeidx}

\theoremstyle{definition}
\newtheorem{definition}{Definition}[section]

\newtheorem{example}[definition]{Example}

\newtheorem{remark}[definition]{Remark}

\theoremstyle{plain}
\newtheorem{theorem}[definition]{Theorem}
\newtheorem{lemma}[definition]{Lemma}
\newtheorem{proposition}[definition]{Proposition}
\newtheorem{corollary}[definition]{Corollary}

\def\N{{\mathbb N}}

\def\V{{\cal V}}

\def\h{\vspace*{7pt}\noindent \textcircled{\em H}\quad }
\def\hnosp{\noindent \textcircled{\em H}\quad }

\def\hyplA{\textcircled{$\lambda $}$_A$}
\def\hyplB{\textcircled{$\lambda $}$_B$}
\def\hypmuA{\textcircled{$\mu $}$_A$}
\def\hypmuB{\textcircled{$\mu $}$_B$}
\def\hypcpA{\textcircled{$\kappa $}$_A$}
\def\hypcpB{\textcircled{$\kappa $}$_B$}

\begin{document}
\title{Congruence Extensions in Congruence--modular Varieties}
\author{George GEORGESCU$^1$, Leonard KWUIDA$^2$ and Claudia MURE\c SAN$^3$\thanks{Corresponding author.}\\ 
$^1$,$^3${\small University of Bucharest, Faculty of Mathematics and Computer Science}\\
$^2${\small Bern University of Applied Sciences, School of Business}\\ 
$^1${\small georgescu.capreni@yahoo.com;}\\
$^2${\small leonard.kwuida@bfh.ch;}\\
$^3${\small cmuresan@fmi.unibuc.ro, claudia.muresan@unibuc.ro}}
\date{\today }
\maketitle 

\begin{abstract} We investigate from an algebraic and topological point of view the minimal prime spectrum of a universal algebra, considering the prime congruences w.r.t. the term condition commutator. Then we use the topological structure of the minimal prime spectrum to study extensions of universal algebras that generalize certain types of ring extensions. Our results hold for semiprime members of semi--degenerate congruence--modular varieties, as well as semiprime algebras whose term condition commutators are commutative and distributive w.r.t. arbitrary joins and satisfy certain conditions on compact congruences, even if those algebras do not generate congruence--modular varieties.
     

\end{abstract}

\section{Introduction}
\label{introduction}

Inspired by group theory and initially developped in \cite{fremck} for congruence--modular varieties, commutator theory has led to the solving of many deep universal algebra problems; it has been later extended by adopting various definitions for the commutator, all of which collapse to the modular commutator in this congruence--modular case.

The congruence lattices of members of congruence--modular varieties, endowed with the modular commutator, form commutator lattices, in which we can introduce the prime elements w.r.t. the commutator operation. For the purpose of not restricting to this congruence--modular setting, we have introduced the notion of a prime congruence through the term condition commutator. Under certain conditions for this commutator operation which do not have to be satisfied throughout a whole variety, the thus defined set of the prime congruences of an algebra becomes a topological space when endowed to a generalization of the Zariski topology from commutative rings \cite{kap,lam}. For members of semi--degenerate congruence--modular varieties, this topological space has strong properties \cite{agl}, some of which extend to more general cases.

The first goal of this paper is to study the topology this generalization induces on the antichain of the minimal prime congruences of an algebra whose term condition commutator satisfies certain conditions, all of which hold in any member of a semi--degenerate congruence--modular variety.

The second goal of our present work is the study of certain types of extensions of algebras with ''well--behaved'' commutators, meaning term condition commutators that behave like the modular commutator, generalizing results on ring extensions from \cite{bhadremcgov,picavet}.

In Section \ref{preliminaries} we recall some results on congruence lattices and the term condition commutator, as well as the particular case of the modular commutator, along with the prime and minimal prime spectra of congruences of an algebra with ''well--behaved'' commutators, where the prime congruences are defined w.r.t. the commutator operation, as well as the prime and minimal spectra of ideals of a bounded distributive lattice. The following sections are dedicated to our new results.

Section \ref{rescglat} contains arithmetical properties of commutator lattices of congruences and annihilators w.r.t. the commutator in such lattices, derived from the residuation operation and its associated negation introduced through these annihilators.

In Section \ref{minspec} we obtain several algebraic properties of the minimal prime spectrum of congruences, including a characterization of minimal prime congruences through their behavior w.r.t. the negations of congruences, obtained in two different cases from a corresponding characterization of minimal prime ideals of bounded distributive lattices.

In Section \ref{2topminspec} we study the Stone (also called spectral) and the flat (also called inverse) topology on the minimal prime spectrum of congruences of an algebra, establish homeomorphisms between these and the corresponding topologies on the minimal prime spectrum of ideals of the reticulation of that algebra (see \cite{retic} for the construction of the reticulation in the universal algebra setting) and obtain necessary and sufficient conditions for these two topologies to coincide.

In Section \ref{algext}, starting from the study of ring extensions in \cite{bhadremcgov,picavet}, we define certain classes of extensions of universal algebras that generalize corresponding classes of ring extensions: $m$--extensions, 
rigid, quasirigid and weak rigid extensions, $r$--extensions and quasi/weak $r$--extensions, $r^{\ast }$--extensions and quasi/weak $r^{\ast }$--extensions, and, generalizing results from \cite{bhadremcgov,picavet}, we obtain relations between these types of extensions, characterizations for these kinds of extensions and topological properties of the minimal prime spectra of the universal algebras that form such extensions.

\section{Preliminaries}
\label{preliminaries}

We refer the reader to \cite{agl,bur,gralgu,koll} for a further study of the following notions from universal algebra, to \cite{bal,blyth,cwdw,gratzer} for the lattice--theoretical ones, to \cite{agl,fremck,koll,ouwe} for the results on commutators and to \cite{agl,cze,cze2,gulo,euadm,joh} for the Stone topologies.

All algebras will be nonempty and they will be designated by their underlying sets; by {\em trivial algebra} we mean one--element algebra. For brevity, we denote by $A\cong B$ the fact that two algebras $A$ and $B$ of the same type are isomorphic. 

$\N $ denotes the set of the natural numbers, $\N ^*=\N \setminus \{0\}$, and, for any $a,b\in \N $, we denote by $\overline{a,b}=\{n\in \N \ |\ a\leq n\leq b\}$ the interval in the lattice $(\N ,\leq )$ bounded by $a$ and $b$, where $\leq $ is the natural order\footnote{This is to differentiate from the notation for commutators}. Let $M$, $N$ be sets and $S\subseteq M$. Then ${\cal P}(M)$ denotes the set of the subsets of $M$ and $({\rm Eq}(M),\vee ,\cap ,\Delta _M=\{(x,x)\ |\ x\in M\},\nabla _M=M^2)$ is the bounded lattice of the equivalences on $M$. We denote by $i_{S,M}:S\rightarrow M$ the inclusion map and by $id_M=i_{M,M}$ the identity map of $M$. For any function $f:M\rightarrow N$, we denote by ${\rm Ker}(f)$ the kernel of $f$, by $f$ the direct image of $f^2=f\times f$ and by $f^*$ the inverse image of $f^2$. 

For any poset $P$, $Max(P)$ and $Min(P)$ will denote the set of the maximal elements and that of the minimal elements of $P$, respectively. Whenever we refer to posets of congruences of an algebra or ideals of a lattice, the order will be the set inclusion.

Let $L$ be a lattice. Then ${\rm Cp}(L)$ and ${\rm Mi}(L)$ denote the set of the compact elements and that of the meet--irreducible elements of $L$, respectively. ${\rm Id}(L)$, ${\rm PId}(L)$ and ${\rm Spec}_{\rm Id}(L)$ denote the sets of the ideals, principal ideals and prime ideals of $L$, respectively. We denote by ${\rm Min}_{\rm Id}(L)=Min({\rm Spec}_{\rm Id}(L))$: the set of the minimal prime ideals of $L$. Let $U\subseteq L$ and $u\in L$. Then $[U)_L$ and $[u)_L$ denote the filters of $L$ generated by $U$ and by $u$, respectively, while $(U]_L$ and $(u]_L$ denote the ideals of $L$ generated by $U$ and by $u$, respectively. If $L$ has a $0$, then ${\rm Ann}_L(U)=\{a\in L\ |\ (\forall \, x\in U)\, (a\wedge x=0)\}$ is the annihilator of $U$ and we denote by ${\rm Ann}_L(u)={\rm Ann}_L(\{u\})$ the annihilator of $u$. The subscript $L$ will be eliminated from these notations when the lattice $L$ is clear from the context. Note that, if $L$ has a $0$ and it is distributive, then all annihilators in $L$ are ideals of $L$.


Recall that a {\em frame} is a complete lattice with the meet distributive w.r.t. arbitrary joins.

Throughout this paper, by {\em functor} we mean covariant functor. ${\cal B}$ denotes the functor from the variety of bounded distributive lattices to the variety of Boolean algebras which takes each bounded distributive lattice to its Boolean center and every morphism in the former variety to its restriction to the Boolean centers. If $L$ is a bounded lattice, then we denote by ${\cal B}(L)$ the set of the complemented elements of $L$ even if $L$ is not distributive.

\h Throughout the rest of this paper: $\tau $ will be a universal algebras signature, $\V $ a variety of $\tau $--algebras and $A$ an arbitrary member of $\V $.\vspace*{7pt}

Everywhere in this paper, we will mark global assumptions as above, for better visibility.

Unless mentioned otherwise, by {\em morphism} we mean $\tau $--morphism.

${\rm Con}(A)$, ${\rm Max}(A)$, ${\rm PCon}(A)$ and ${\cal K}(A)$ denote the sets of the congruences, maximal (proper) congruences, principal congruences and finitely generated congruences of $A$, respectively; note that ${\cal K}(A)={\rm Cp}({\rm Con}(A))$. ${\rm Max}(A)$ is called the {\em maximal spectrum} of $A$. For any $X\subseteq A^2$ and any $a,b\in A$, $Cg_A(X)$ will be the congruence of $A$ generated by $X$ and we shall denote by $Cg_A(a,b)=Cg_A(\{(a,b)\})$.

For any $\theta \in {\rm Con}(A)$, $p_{\theta }:A\rightarrow A/\theta $ will be the canonical surjective morphism; given any $X\in A\cup A^2\cup {\cal P}(A)\cup {\cal P}(A^2)$, we denote by $X/\theta =p_{\theta }(X)$. Note that ${\rm Ker}(p_{\theta })=\theta $ for any $\theta \in {\rm Con}(A)$, and that $Cg_A(Cg_S(X))=Cg_A(X)$ for any subalgebra $S$ of $A$ and any $X\subseteq S^2$.

If $L$ is a distributive lattice, so that we have the canonical lattice embedding $\iota _L:{\rm Id}(L)\rightarrow {\rm Con}(L)$, then we will denote, for every $I\in {\rm Id}(L)$, by $\pi _I=p_{\iota _{\cal L}(I)}:L\rightarrow L/I$.

\h Throughout the rest of this paper, $B$ will be a member of $\V $ and $f:A\rightarrow B$ a morphism.\vspace*{7pt}

Recall that, for any $\alpha \in {\rm Con}(A)$ and any $\beta \in {\rm Con}(B)$, we have $f^*(\beta )\in [{\rm Ker}(f))\subseteq {\rm Con}(A)$, $f(f^*(\beta ))=\beta \cap f(A^2)\subseteq \beta $ and $\alpha \subseteq f^*(f(\alpha ))$; if $\alpha \in [{\rm Ker}(f))$, then 
$f(\alpha )\in {\rm Con}(f(A))$ and $f^*(f(\alpha ))=\alpha $. Hence $\theta \mapsto f(\theta )$ is a lattice isomorphism from $[{\rm Ker}(f))$ to ${\rm Con}(f(A))$, having $f^*$ as inverse, and thus it sets an order isomorphism from ${\rm Max}(A)\cap [{\rm Ker}(f))$ to ${\rm Max}(f(A))$. In particular, for any $\theta \in {\rm Con}(A)$, the map $\alpha \mapsto \alpha /\theta $ is an order isomorphism from $[\theta )$ to ${\rm Con}(A/\theta )$.

\begin{lemma}{\rm \cite[Lemma $1.11$]{bak}, \cite[Proposition $1.2$]{urs5}} For any $X\subseteq A^2$ and any $\alpha ,\theta \in {\rm Con}(A)$:\begin{itemize}
\item $f(Cg_A(X)\vee {\rm Ker}(f))=Cg_{f(A)}(f(X))$, so $Cg_B(f(Cg_A(X)))=Cg_B(f(X))$ and $(Cg_A(X)\vee \theta )/\theta =Cg_{A/\theta }(X/\theta )$;
\item in particular, $f(\alpha \vee {\rm Ker}(f))=Cg_{f(A)}
(f(\alpha ))$, so $(\alpha \vee \theta )/\theta =Cg_{A/\theta }(\alpha /\theta )$.\end{itemize}\label{fcg}\end{lemma}

For any nonempty family $(\alpha _i)_{i\in I}\subseteq [{\rm Ker}(f))$, we have, in ${\rm Con}(f(A))$: $\displaystyle f(\bigvee _{i\in I}\alpha _i)=\bigvee _{i\in I}f(\alpha _i)$. Indeed, by Lemma \ref{fcg}, $\displaystyle f(\bigvee _{i\in I}\alpha _i)=f(Cg_A(\bigcup _{i\in I}\alpha _i))=Cg_{f(A)}(f(\bigcup _{i\in I}\alpha _i))=Cg_{f(A)}(\bigcup _{i\in I}f(\alpha _i))=\bigvee _{i\in I}f(\alpha _i)$.

We denote by $f^{\bullet }:{\rm Con}(A)\rightarrow {\rm Con}(B)$ the map defined by $f^{\bullet }(\alpha )=Cg_B(f(\alpha ))$ for all $\alpha \in {\rm Con}(A)$. By the above, if $f$ is surjective, then $f^{\bullet }\mid _{[{\rm Ker}(f))}:[{\rm Ker}(f))\rightarrow {\rm Con}(B)$ is the inverse of the lattice isomorphism $f^*\mid _{{\rm Con}(B)}:{\rm Con}(B)\rightarrow [{\rm Ker}(f))$.

We use the following definition from \cite{mcks} for the {\em term condition commutator}: let $\alpha ,\beta \in {\rm Con}(A)$. For any $\mu \in {\rm Con}(A)$, by $C(\alpha ,\beta ;\mu )$ we denote the fact that the following condition holds: for all $n,k\in \N $ and any term $t$ over $\tau $ of arity $n+k$, if $(a_i,b_i)\in \alpha $ for all $i\in \overline{1,n}$ and $(c_j,d_j)\in \beta $ for all $j\in \overline{1,k}$, then $(t^A(a_1,\ldots ,a_n,c_1,\ldots ,c_k),t^A(a_1,\ldots ,a_n,\linebreak d_1,\ldots ,d_k))\in \mu $ iff $(t^A(b_1,\ldots ,b_n,c_1,\ldots ,c_k),t^A(b_1,\ldots ,b_n,d_1,\ldots ,d_k))\in \mu $. We denote by $[\alpha ,\beta ]_A=\bigcap \{\mu \in {\rm Con}(A)\ |\ C(\alpha ,\beta ;\mu )\}$; we call $[\alpha ,\beta ]_A$ the {\em commutator of $\alpha $ and $\beta $} in $A$. The operation $[\cdot ,\cdot ]_A:{\rm Con}(A)\times {\rm Con}(A)\rightarrow {\rm Con}(A)$ is called the {\em commutator of $A$}.

By \cite{fremck}, if $\V $ is congruence--modular, then, for each member $M$ of $\V $, $[\cdot ,\cdot ]_M$ is the unique binary operation on ${\rm Con}(M)$ such that, for all $\alpha ,\beta \in {\rm Con}(M)$, $[\alpha ,\beta ]_M=\min \{\mu \in {\rm Con}(M)\ |\ \mu \subseteq \alpha \cap \beta $ and, for any member $N$ of $\V $ and any surjective morphism $h:M\rightarrow N$ in $\V $, $\mu \vee {\rm Ker}(h)=h^*([h(\alpha \vee {\rm Ker}(h)),h(\beta \vee {\rm Ker}(h))]_N)\}$. Therefore, if $\V $ is congruence--modular, $\alpha ,\beta ,\theta \in {\rm Con}(A)$ and $f$ is surjective, then $[f(\alpha \vee {\rm Ker}(f)),f(\beta \vee {\rm Ker}(f))]_B=f([\alpha ,\beta ]_A\vee {\rm Ker}(f))$, in particular $[(\alpha \vee \theta )/\theta ,(\beta \vee \theta )/\theta]_{A/\theta }=([\alpha ,\beta ]_A\vee \theta )/\theta $, hence, if $\theta \subseteq \alpha \cap \beta $, then $[\alpha /\theta ,\beta /\theta ]_{A/\theta }=([\alpha ,\beta ]_A\vee \theta )/\theta $, and, if, moreover, $\theta \subseteq [\alpha ,\beta ]_A$, then $[\alpha /\theta ,\beta /\theta ]_{A/\theta }=[\alpha ,\beta ]_A/\theta $.

By \cite[Lemma 4.6,Lemma 4.7,Theorem 8.3]{mcks}, the commutator is smaller than the intersection and increasing in both arguments; if $\V $ is congruence--modular, then the commutator is also commutative and distributive in both arguments w.r.t. arbitrary joins.

Hence, if $\V $ is congruence--modular and the commutator of $A$ coincides to the intersection of congruences, then ${\rm Con}(A)$ is a frame, in particular it is distributive.

Therefore, if $\V $ is congruence--modular and the commutator coincides to the intersection in each member of $\V $, then $\V $ is congruence--distributive. By \cite{bj}, the converse holds, as well: if $\V $ is congruence--distributive, then, in each member of $\V $, the commutator coincides to the intersection of congruences.

For any $\alpha ,\beta \in {\rm Con}(A)$, we denote by $[\alpha ,\beta ]_A^1=[\alpha ,\beta ]_A$ and, for any $n\in \N ^*$, by $[\alpha ,\beta ]_A^{n+1}=[[\alpha ,\beta ]_A^n,[\alpha ,\beta ]_A^n]_A$.

Recall that $A$ is called an {\em Abelian algebra} iff $[\nabla _A,\nabla _A]_A=\Delta _A$.

Recall that, for any nonempty family $(M_i)_{i\in I}$ of members of $\V $, the {\em skew congruences} of the direct product $\displaystyle \prod _{i\in I}M
_i$ are the elements of $\displaystyle {\rm Con}(\prod _{i\in I}M_i)\setminus \prod _{i\in I}{\rm Con}(M_i)$, where $\displaystyle \prod _{i\in I}{\rm Con}(M_i)=\{\prod _{i\in I}\alpha _i\ |\ (\forall \, i\in I)\, (\alpha _i\in {\rm Con}(M_i))\}$, with the direct product of binary relations having the usual definition.

By \cite[Theorem 8.5, p. 85]{fremck}, if $\V $ is congruence--modular, then the following are equivalent:\begin{itemize}
\item $\V $ contains no nontrivial Abelian algebras, that is, for any nontrivial algebra $M$ from $\V $, $[\nabla _M,\nabla _M]_M\neq \Delta _M$;
\item for any algebra $M$ from $\V $, $[\nabla _M,\nabla _M]_M=\nabla _M$;
\item for any algebra $M$ from $\V $ and any $\theta \in {\rm Con}(M)$, $[\theta ,\nabla _M]_M=\theta $;
\item $\V $ has no skew congruences, that is, for any algebras $M$ and $N$ from $\V $, ${\rm Con}(M\times N)=\{\theta \times \zeta \ |\ \theta \in {\rm Con}(M),\zeta \in {\rm Con}(N)\}$.\end{itemize}

Recall that $\V $ is said to be {\em semi--degenerate} iff no nontrivial algebra in $\V $ has one--element subalgebras. Recall from \cite{koll} that, if $\V $ is congruence--modular, then the following are equivalent:\begin{itemize}
\item $\V $ is semi--degenerate;
\item for all members $M$ of $\V $, $\nabla _M\in {\cal K}(M)$.\end{itemize}

By \cite[Lemma 5.2]{agl}, the equivalences in \cite[Theorem 8.5, p. 85]{fremck} recalled above and the fact that, in congruence--distributive varieties, the commutator coincides to the intersection, we have: if $\V $ is either congruence--distributive or both congruence--mo\-du\-lar and semi--degenerate, then $\V $ has no skew congruences.

If $[\cdot ,\cdot ]_A$ is distributive w.r.t. the join, in particular if $\V $ is congruence--modular, then, if $A$ has {\em principal commutators}, that is its set ${\rm PCon}(A)$ of principal congruences is closed w.r.t. the commutator, then  $A$ has {\em compact commutators}, that is its set ${\cal K}(A)$ of compact congruences is closed w.r.t. the commutator. Consequently, if the commutator of $A$ equals the intersection of congruences, in particular if $\V $ is congruence--distributive, then, if $A$ has the {\em principal intersection property (PIP)}, that is ${\rm PCon}(A)$ is closed w.r.t. the intersection, then $A$ has the {\em compact intersection property (CIP)}, that is ${\cal K}(A)$ is closed w.r.t. the intersection.

Recall that a {\em prime congruence} of $A$ is a proper congruence $\phi $ of $A$ such that, for any $\alpha ,\beta \in {\rm Con}(A)$, if $[\alpha ,\beta ]_A\subseteq \phi $, then $\alpha \subseteq \phi $ or $\beta \subseteq \phi $ \cite{fremck}. It actually suffices that we enforce this condition for principal congruences $\alpha $, $\beta $ of $A$:

\begin{lemma}{\rm \cite{gulo,retic}} A proper congruence $\phi $ of $A$ is prime iff for any $\alpha ,\beta \in {\rm PCon}(A)$, if $[\alpha ,\beta ]_A\subseteq \phi $, then $\alpha \subseteq \phi $ or $\beta \subseteq \phi $.\label{charspec}\end{lemma}

We denote by ${\rm Spec}(A)$ the {\em (prime) spectrum} of $A$, that is the set of the prime congruences of $A$. Recall that ${\rm Spec}(A)$ is not necessarily nonempty. However, by \cite[Theorem $5.3$]{agl}, if the commutator of $A$ is distributive w.r.t. the join of congruences, $\nabla _A\in {\cal K}(A)$ and $[\nabla _A,\nabla _A]_A=\nabla _A$, in particular if $\V $ is congruence--modular and semi--de\-ge\-ne\-rate, then:\begin{itemize}
\item ${\rm Max}(A)\subseteq {\rm Spec}(A)$;
\item any proper congruence of $A$ is included in a maximal and thus a prime congruence of $A$;
\item hence ${\rm Max}(A)$ and thus ${\rm Spec}(A)$ is nonempty whenever $A$ is nontrivial.\end{itemize}

For each $\theta \in {\rm Con}(A)$, we denote by $V_A(\theta )={\rm Spec}(A)\cap [\theta )$ and by $D_A(\theta )={\rm Spec}(A)\setminus V_A(\theta )={\rm Spec}(A)\setminus [\theta )$. For each $X\subseteq A^2$ and all $a,b\in A$, we denote by $V_A(X)=V_A(Cg_A(X))$, $D_A(X)=D_A(Cg_A(X))$, $V_A(a,b)=V_A(Cg_A(a,b))$ and $D_A(a,b)=D_A(Cg_A(a,b))$.

For any $\theta \in {\rm Con}(A)$, we denote by $\rho _A(\theta )=\bigcap V_A(\theta )$ and call this congruence the {\em radical} of $\theta $. We denote by ${\rm RCon}(A)=\{\rho _A(\theta )\ |\ \theta \in {\rm Con}(A)\}=\{\theta \in {\rm Con}(A)\ |\ \rho _A(\theta )=\theta \}$. We call the elements of ${\rm RCon}(A)$ the {\em radical congruences} of $A$. Obviously, any prime congruence of $A$ is radical.

By \cite[Lemma~$1.6$, Proposition~$1.2$]{agl}, if the commutator of $A$ is commutative and distributive w.r.t. arbitrary joins, in particular if $\V $ is congruence--modular, then:\begin{enumerate}
\item\label{radsemipr} a congruence $\theta $ of $A$ is radical iff it is {\em semiprime}, that is, for any $\alpha \in {\rm Con}(A)$, if $[\alpha ,\alpha ]_A\subseteq \theta $, then $\alpha \subseteq \theta $;
\item\label{prmirad} hence ${\rm Spec}(A)={\rm Mi}({\rm Con}(A))\cap {\rm RCon}(A)$.\end{enumerate}

$A$ is called a {\em semiprime algebra} iff $\rho _A(\Delta _A)=\Delta _A$. By statement (\ref{radsemipr}) above, if the commutator of $A$ equals the intersection, in particular if $\V $ is congruence--distributive, then ${\rm RCon}(A)={\rm Con}(A)$, thus $A$ is semiprime.

Let us denote by ${\cal S}_{\rm Spec}(A)=\{D_A(\theta )\ |\ \theta \in {\rm Con}(A)\}$. If the commutator of $A$ is commutative and distributive w.r.t. arbitrary joins, in particular if $\V $ is congruence--modular, then, by \cite{agl,gulo,retic}, ${\cal S}_{\rm Spec}(A)$ is a topology on ${\rm Spec}(A)$, called the {\em Stone topology} or the {\em spectral topology}, which satisfies, for all $\alpha ,\beta \in {\rm Con}(A)$ and any family $(\alpha _i)_{i\in I}\subseteq {\rm Con}(A)$:\begin{itemize}
\item $D_A(\alpha )\subseteq D_A(\beta )$ iff $V_A(\alpha )\supseteq V_A(\beta )$ iff $\rho_A(\alpha )\subseteq \rho_A(\beta )$;
\item thus $D_A(\alpha )=D_A(\beta )$ iff $V_A(\alpha )=V_A(\beta )$ iff $\rho_A(\alpha )=\rho_A(\beta )$;
\item clearly, $\alpha \subseteq \beta $ implies $\rho_A(\alpha )\subseteq \rho_A(\beta )$;
\item clearly, $\alpha \subseteq \rho_A(\alpha )$, thus $\rho_A(\alpha )=\Delta _A$ implies $\alpha =\Delta _A$;
\item $D_A(\nabla _A)={\rm Spec}(A)=V_A(\Delta _A)$ and $D_A(\Delta _A)=\emptyset =V_A(\nabla _A)$;
\item if $A$ is semiprime, then: $D_A(\alpha )=\emptyset $ iff $V_A(\alpha )={\rm Spec}(A)$ iff $\rho_A(\alpha )=\Delta _A$ iff $\alpha =\Delta _A$;
\item if $\nabla _A\in {\cal K}(A)$ and $[\nabla _A,\nabla _A]_A=\nabla _A$, in particular if $\V $ is congruence--modular and semi--degenerate, then: $D_A(\alpha )={\rm Spec}(A)$ iff $V_A(\alpha )=\emptyset $ iff $\rho_A(\alpha )=\nabla _A$ iff $\alpha =\nabla _A$;
\item $D_A([\alpha ,\beta ]_A)=D_A(\alpha \cap \beta )=D_A(\alpha )\cap D_A(\beta )$ and $D_A(\alpha \vee \beta )=D_A(\alpha )\cup D_A(\beta )$, thus $V_A([\alpha ,\beta ]_A)=V_A(\alpha \cap \beta )=V_A(\alpha )\cup V_A(\beta )$, $V_A(\alpha \vee \beta )=V_A(\alpha )\cap V_A(\beta )$, $\rho_A([\alpha ,\beta ]_A)=\rho_A(\alpha \cap \beta )=\rho_A(\alpha )\cap \rho_A(\beta )$ and $\rho_A(\alpha \vee \beta )=\rho_A(\alpha )\vee \rho_A(\beta )$;
\item $\displaystyle D_A(\bigvee _{i\in I}\alpha _i)=D_A(\bigcup _{i\in I}\alpha _i)=\bigcup _{i\in I}D_A(\alpha _i)$, thus $\displaystyle V_A(\bigvee _{i\in I}\alpha _i)=V_A(\bigcup _{i\in I}\alpha _i)=\bigcap _{i\in I}V_A(\alpha _i)$ and $\displaystyle \rho_A(\bigvee _{i\in I}\alpha _i)=\rho_A(\bigcup _{i\in I}\alpha _i)=\rho_A(\bigcup _{i\in I}\rho_A(\alpha _i))=\bigvee _{i\in I}\rho_A(\alpha _i)$;
\item hence, for any $\theta \in {\rm Con}(A)$, $\displaystyle V_A(\theta )=\bigcap _{(a,b)\in \theta }V_A(a,b)$ and $\displaystyle D_A(\theta )=\bigcup _{(a,b)\in \theta }D_A(a,b)$, therefore the Stone topology ${\cal S}_{\rm Spec}(A)$ has $\{D_A(a,b)\ |\ a,b\in A\}$ as a basis.\end{itemize}

Assume that $[\cdot ,\cdot ]_A$ is commutative and distributive w.r.t. arbitrary joins. Assume, moreover, that ${\rm Max}(A)\subseteq {\rm Spec}(A)$, which holds if $\nabla _A\in {\cal K}(A)$ and $[\nabla _A,\nabla _A]_A=\nabla _A$. All of this holds in the particular case when $\V $ is congruence--modular and semi--de\-ge\-ne\-rate. Then the Stone topology ${\cal S}_{\rm Spec}(A)$ on ${\rm Spec}(A)$ induces {\em Stone} or {\em spectral topology} on ${\rm Max}(A)$: ${\cal S}_{\rm Max}(A)=\{D_A(\theta )\cap {\rm Max}(A)\ |\ \theta \in {\rm Con}(A)\}$, having $\{D_A(a,b)\cap {\rm Max}(A)\ |\ a,b\in A\}$ as a basis.

In the same way, but replacing congruences with ideals, one defines the Stone topology on the set of prime ideals and that of maximal ideals of a bounded distributive lattice.

In \cite{gulo,euadm}, we have called $f$ an {\em admissible morphism} iff $f^*({\rm Spec}(B))\subseteq {\rm Spec}(A)$. Recall from \cite{agl} that, if $\V $ is congruence--modular, then the map $\alpha \mapsto f(\alpha )$ is an order isomorphism from ${\rm Spec}(A)\cap [{\rm Ker}(f))$ to ${\rm Spec}(f(A))$, thus to ${\rm Spec}(B)$ if $f$ is surjective, case in which this map coincides with $f^{\bullet }$ and $f^*$ is its inverse, hence $f$ is admissible.

\begin{remark} By the above, if $\V $ is congruence--modular and $f$ is s]urjective, then:\begin{itemize}
\item for all $\alpha \in {\rm Con}(A)$, $f(V_A(\alpha ))=V_B(f(\alpha ))$ and $f(D_A(\alpha ))=D_B(f(\alpha ))$;
\item in particular, for all $a,b\in A$, $f(V_A(a,b))=V_B(f(a),f(b))$ and $f(D_A(a,b))=D_B(f(a),f(b))$,
\end{itemize}

\noindent thus, since $f=f^{\bullet }=(f^*)^{-1}$, the map $f^*\mid _{{\rm Spec}(B)}:{\rm Spec}(B)\rightarrow {\rm Spec}(A)$ is continuous w.r.t. the Stone topologies.\end{remark}

A subset $S$ of $A^2$ is called an {\em $m$--system} for $A$ iff, for all $a,b,c,d\in A$, if $(a,b),(c,d)\in S$, then $[Cg_A(a,b),\linebreak Cg_A(c,d)]_A\cap S\neq \emptyset $. For instance, any congruence of $A$ is an $m$--system. Also:

\begin{remark}{\rm \cite{gulo,retic}} If $\phi \in {\rm Spec}(A)$, then $\nabla _A\setminus \phi $ is an $m$--system in $A$.\label{msist}\end{remark}

\begin{lemma}{\rm \cite{agl}} Let $S$ be an $m$--system in $A$ and $\alpha \in {\rm Con}(A)$ such that $\alpha \cap S=\emptyset $. If the commutator of $A$ is distributive w.r.t. the join, in particular if $\V $ is congruence--modular, then:\begin{itemize}
\item ${\rm Max}\{\theta \in {\rm Con}(A)\ |\ \alpha \subseteq \theta ,\theta \cap S=\emptyset \}\subseteq {\rm Spec}(A)$, in particular, for the case $\alpha =\Delta _A$, ${\rm Max}\{\theta \in {\rm Con}(A)\ |\ \theta \cap S=\emptyset \}\subseteq {\rm Spec}(A)$;
\item if $\nabla _A\in {\cal K}(A)$, in particular in $\V $ is congruence--modular and semi--degenerate, then the set ${\rm Max}\{\theta \in {\rm Con}(A)\ |\ \alpha \subseteq \theta ,\theta \cap S=\emptyset \}$ is nonempty, in particular ${\rm Max}\{\theta \in {\rm Con}(A)\ |\ \theta \cap S=\emptyset \}$ is nonempty.\end{itemize}\label{msistspec}\end{lemma}

\begin{remark} If $\V $ is congruence--modular, $S$ is an $m$--system in $A$ and $f$ is surjective, then $f(S)$ is an $m$--system in $B$.\end{remark}

We denote by ${\rm Min}(A)={\rm Min}({\rm Spec}(A),\subseteq )$. Recall that ${\rm Min}(A)$ is called the {\em minimal prime spectrum} of $A$ and its elements are called {\em minimal prime congruences} of $A$.

Now assume that the commutator of $A$ is commutative and distributive w.r.t. arbitrary joins, which holds if $\V $ is congruence--modular. Then, by \cite[Proposition $5.9$]{stcommlat}, if we define a binary relation $\equiv _A$ on ${\rm Con}(A)$ by: for any $\alpha ,\beta \in {\rm Con}(A)$, $\alpha \equiv _A\beta $ iff $\rho _A(\alpha )=\rho _A(\beta )$, then $\equiv _A$ is a lattice congruence of ${\rm Con}(A)$ that preserves arbitrary joins such that ${\rm Con}(A)/\!\equiv _A$ is a frame; see also \cite{retic}.

Following the notations from \cite{stcommlat}, if $(L,[\cdot ,\cdot ])$ is a {\em commutator lattice}, that is a complete lattice $L$ endowed with a binary operation $[\cdot ,\cdot ]$ which is commutative, smaller than the meet and distributive w.r.t. arbitrary joins \cite{cze2,gal}, then we denote by ${\rm Spec}_L$ the set of the prime elements of $L$ w.r.t. the commutator $[\cdot ,\cdot ]$, by ${\rm Min}_L=Min({\rm Spec}_L)$ the set of the minimal prime elements of $L$ and by $R(L)$ the set of the radical elements of $L$, that is the meets of subsets of ${\rm Spec}_L$. 

If $(L,[\cdot ,\cdot ])$ is a commutator lattice, $u\in L$ and $U\subseteq L$, then, in order to differentiate between annihilators w.r.t. to the meet and those w.r.t. the commutator, we will use the following notations: the annihilator of $U$ in $(L,[\cdot ,\cdot ])$ is ${\rm Ann}_{(L,[\cdot ,\cdot ])}(U)=\{a\in L\ |\ (\forall \, x\in U)\, ([a,x]=0)\}$ and the annihilator of $u$ in $(L,[\cdot ,\cdot ])$ is ${\rm Ann}_{(L,[\cdot ,\cdot ])}(u)={\rm Ann}_{(L,[\cdot ,\cdot ])}(\{u\})$.

Recall from \cite{stcommlat} that $L$ is a frame iff its commutator $[\cdot ,\cdot ]$ equals the meet, case in which the annihilators in $(L,[\cdot ,\cdot ])$ coincide with those w.r.t. the meet and ${\rm Spec}_L$ is exactly the set of the meet--prime elements of $L$, thus ${\rm Spec}_L={\rm Mi}(L)$ since $L$ is distributive.

\section{On the Residuated Structure of the Lattice of Congruences}
\label{rescglat}

\hnosp Throughout this section, we will assume that the commutator of $A$ is commutative and distributive w.r.t. arbitrary joins, which holds if $\V $ is congruence--modular.\vspace*{7pt}

See {\rm \cite{retic}} for the next results. Let $\alpha ,\beta ,\gamma ,\theta \in {\rm Con}(A)$ and $n\in \N ^*$, arbitrary.

An induction argument shows that:\begin{itemize}
\item\label{commn1} $[\alpha ,\beta ]_A^{n+1}=[[\alpha ,\beta ]_A^n,[\alpha ,\beta ]_A^n]_A$;
\item\label{commn0} $\rho _A([\alpha ,\beta ]_A^n)=\rho _A([\alpha ,\beta ]_A)=\rho _A(\alpha \cap \beta )=\rho _A(\alpha )\cap \rho _A(\beta )$.\end{itemize}

If $A$ is semiprime, then $\rho _A(\theta )=\Delta _A$ iff $\theta =\Delta _A$, therefore:\begin{itemize}
\item since $\rho _A([\alpha ,\alpha ]_A^n)=\rho _A([\alpha ,\alpha ]_A)=\rho _A(\alpha )$, it follows that: $\rho _A([\alpha ,\alpha ]_A^n)=\Delta _A$ iff $\alpha =\Delta _A$;
\item since $\rho _A([\alpha ,[\beta ,\gamma ]_A]_A)=\rho _A(\alpha \cap \beta \cap \gamma )=\rho _A([[\alpha ,\beta ]_A,\gamma ]_A)$, it follows that: $[\alpha ,[\beta ,\gamma ]_A]_A=\Delta _A$ iff $\alpha \cap \beta \cap \gamma =\Delta _A$ iff $[[\alpha ,\beta ]_A,\gamma ]_A=\Delta _A$.\end{itemize}

If $\V $ is congruence--modular and $f$ is surjective, then, for any $X,Y\in {\cal P}(A^2)$ and any $a,b,c,d\in A$:\begin{itemize}
\item\label{totcommn2} $[f(\alpha \vee {\rm Ker}(f)),f(\beta \vee {\rm Ker}(f))]_B^n=f([\alpha ,\beta ]_A^n\vee {\rm Ker}(f))$, in particular $[(\alpha \vee \theta )/\theta ,(\beta \vee \theta )/\theta ]_{A/\theta }^n=([\alpha ,\beta ]_A^n\vee \theta )/\theta $;
\item\label{totcommn3} hence $[Cg_{A/\theta }(X/\theta ),Cg_{A/\theta }(Y/\theta )]_{A/\theta }^n=([Cg_A(X),Cg_A(Y)]_A^n\vee \theta )/\theta $, in particular $[Cg_{A/\theta }(a/\theta ,b/\theta ),\linebreak Cg_{c/\theta ,d/\theta }(Y/\theta )]_{A/\theta }^n=([Cg_A(a,b),Cg_A(c,d)]_A^n\vee \theta )/\theta $;
\item ${\rm Spec}(B)=\{\phi /{\rm Ker}(f)\ |\ \phi \in V_A({\rm Ker}(f))\}$, in particular ${\rm Spec}(A/\theta )=\{\phi /\theta \ |\ \phi \in V_A(\theta )\}$.\end{itemize}\label{totcommn}

We denote by $\beta \rightarrow \gamma =\bigvee \{\delta \in {\rm Con}(A)\ |\ [\delta ,\beta ]_A\subseteq \gamma \}$ and $\beta ^{\perp }=\beta \rightarrow \Delta _A$.

Clearly, $\beta \rightarrow \gamma =\bigvee \{\delta \in {\cal K}(A)\ |\ [\delta ,\beta ]_A\subseteq \gamma \}=\bigvee \{\delta \in {\rm PCon}(A)\ |\ [\delta ,\beta ]_A\subseteq \gamma \}$, so $\beta ^{\perp }=\bigvee \{\delta \in {\rm Con}(A)\ |\ [\delta ,\beta ]_A=\Delta _A\}=\bigvee \{\delta \in {\cal K}(A)\ |\ [\delta ,\beta ]_A=\Delta _A\}=\bigvee \{\delta \in {\rm PCon}(A)\ |\ [\delta ,\beta ]_A=\Delta _A\}$.

Note that these operations can be defined for any commutator lattice $(L,[\cdot ,\cdot ])$ by: $b\rightarrow c=\bigvee \{a\in L\ |\ [a,b]\leq c\}$ and $b^{\perp }=b\rightarrow 0=\bigvee \{a\in L\ |\ [a,b]=0\}$ for any $b,c\in L$ and, if $L$ is algebraic, that is compactly generated, then we also have equalities similar to the above.


Since $[\Delta _A,\beta ]_A=\Delta _A\subseteq \gamma $ and, for any non--empty family $(\alpha _i)_{i\in I}$, $[\alpha _i,\beta ]_A\subseteq \gamma $ for all $i\in I$ implies $\displaystyle [\bigvee _{i\in I}\alpha _i,\beta ]_A=\bigvee _{i\in I}[\alpha _i,\beta ]_A\subseteq \gamma $, it follows that:\vspace*{-12pt}$$\beta \rightarrow \gamma =\max \{\delta \in {\rm Con}(A)\ |\ [\delta ,\beta ]_A\subseteq \gamma \},$$$$\mbox{in particular }\beta ^{\perp }=\max \{\delta \in {\rm Con}(A)\ |\ [\delta ,\beta ]_A=\Delta _A\},$$hence, in the particular case in which $A$ is semiprime, $\beta ^{\perp }=\max ({\rm Ann}_{({\rm Con}(A),[\cdot ,\cdot ]_A)}(\beta ))$ and thus ${\rm Ann}_{({\rm Con}(A),[\cdot ,\cdot ]_A)}(\beta )=(\beta ^{\perp }]\in {\rm PId}({\rm Con}(A))$.

The definitions above also show that:$$[\beta ,\beta \rightarrow \gamma ]_A\subseteq \gamma \mbox{, in particular }[\beta ,\beta ^{\perp }]_A=\Delta _A;$$moreover, for all $\delta \in {\rm Con}(A)$:$$[\delta ,\beta ]_A\subseteq \gamma \mbox{ iff }\delta \subseteq \beta \rightarrow \gamma ,\mbox{ in particular: }[\delta ,\beta ]_A=\Delta _A\mbox{ iff }\delta \subseteq \beta ^{\perp }.$$Hence, in the particular case when the commutator of $A$ is associative, $({\rm Con}(A),\vee ,\cap,\rightarrow ,\Delta _A,\nabla _A)$ is a (bounded commutative integral) residuated lattice, in which $\cdot ^{\perp }$ is the negation.

\begin{lemma} If the algebra $A$ is semiprime, then $\theta ^{\perp }\in {\rm RCon}(A)$ for any $\theta \in {\rm Con}(A)$.\label{negrad}\end{lemma}

\begin{proof} Let $\alpha ,\theta \in {\rm Con}(A)$ such that $[\alpha ,\alpha ]_A\subseteq \theta ^{\perp }$. Then, by the above and the fact that $A$ is semiprime, $[[\alpha ,\alpha ]_A,\theta ]_A=\Delta _A$, which is equivalent to $\rho _A([[\alpha ,\alpha ]_A,\theta ]_A)=\Delta _A$, that is $\rho _A(\alpha \cap \theta )=\Delta _A$, that is $\rho _A([\alpha ,\theta ]_A)=\Delta _A$, which means that $[\alpha ,\theta ]_A=\Delta _A$, which in turn is equivalent to $\alpha \subseteq \theta ^{\perp }$. Hence $\theta ^{\perp }$ is a semiprime and thus a radical congruence of $A$.\end{proof}

Since $\displaystyle \theta =\bigvee _{(a,b)\in \theta }Cg_A(a,b)=\bigvee \{\zeta \in {\rm PCon}(A)\ |\ \zeta \subseteq \theta \}=\bigvee \{\zeta \in {\cal K}(A)\ |\ \zeta \subseteq \theta \}$, it follows that:$$\beta \rightarrow \gamma =\bigvee \{\zeta \in {\cal K}(A)\ |\ [\zeta ,\beta ]_A\subseteq \gamma \}=\bigvee \{\zeta \in {\rm PCon}(A)\ |\ [\zeta ,\beta ]_A\subseteq \gamma \},$$$$\mbox{in particular }\beta ^{\perp }=\bigvee \{\zeta \in {\cal K}(A)\ |\ [\zeta ,\beta ]_A=\Delta _A\}=\bigvee \{\zeta \in {\rm PCon}(A)\ |\ [\zeta ,\beta ]_A=\Delta _A\}.$$Let us note that, for all $a,b\in A$, we have $(a,b)\in \beta ^{\perp }$ iff $Cg_A(a,b)\subseteq \beta ^{\perp }$ iff $[Cg_A(a,b),\beta ]_A=\Delta _A$, hence $\beta ^{\perp }=\{(a,b)\in A^2\ |\ [Cg_A(a,b),\beta ]_A=\Delta _A\}$.

For any $X\subseteq A^2$, we denote by $X^{\perp }=\{(a,b)\in A^2\ |\ (\forall \, (x,y)\in X)\, ([Cg_A(a,b),Cg_A(x,y)]_A=\Delta _A)\}$, so that:$$\displaystyle X^{\perp }=\{(a,b)\in A^2\ |\ [Cg_A(a,b),\bigvee _{(x,y)\in X}Cg_A(x,y)]_A=\Delta _A\}=\{(a,b)\in A^2\ |\ [Cg_A(a,b),Cg_A(X)]_A=\Delta _A\}=$$$$\bigvee \{Cg_A(a,b)\ |\ (a,b)\in A^2,[Cg_A(a,b),Cg_A(X)]_A=\Delta _A\}=\bigvee \{\alpha \in {\rm Con}(A)\ |\ [\alpha ,Cg_A(X)]_A=\Delta _A\}=$$$$\max \{\alpha \in {\rm Con}(A)\ |\ [\alpha ,Cg_A(X)]_A=\Delta _A\}=Cg_A(X)^{\perp },$$so this more general notation is consistent with the notation above for the particular case when $X\in {\rm Con}(A)$.

\begin{lemma} For any $\alpha ,\beta ,\theta \in {\rm Con}(A)$:\begin{enumerate}
\item\label{implic1} $\beta \subseteq \alpha \rightarrow \beta $;
\item\label{implic2} $(\alpha \vee \theta )\rightarrow (\beta \vee \theta )=\alpha \rightarrow (\beta \vee \theta )$.
\end{enumerate}\label{implic}\end{lemma}

\begin{proof} (\ref{implic1}) $[\beta ,\alpha ]_A\subseteq \beta \cap \alpha \subseteq \beta $, thus $\beta \subseteq \max \{\zeta \in {\cal K}(A)\ |\ [\zeta ,\alpha ]_A\subseteq \beta \}=\alpha \rightarrow \beta $.

\noindent (\ref{implic2}) For all $\gamma \in {\rm Con}(A)$, we have, since $[\gamma ,\theta ]_A\subseteq \theta \subseteq \beta \vee \theta $: $\gamma \subseteq (\alpha \vee \theta )\rightarrow (\beta \vee \theta )$ iff $[\gamma ,\alpha \vee \theta ]_A\subseteq \beta \vee \theta $ iff $[\gamma ,\alpha ]_A\vee [\gamma ,\theta ]_A\subseteq \beta \vee \theta $ iff $[\gamma ,\alpha ]_A\subseteq \beta \vee \theta $ iff $\gamma \subseteq \alpha \rightarrow (\beta \vee \theta )$. By taking $\gamma =(\alpha \vee \theta )\rightarrow (\beta \vee \theta )$ and then $\gamma =\alpha \rightarrow (\beta \vee \theta )$ in the previous equivalences, we get: $\alpha \rightarrow (\beta \vee \theta )=(\alpha \vee \theta )\rightarrow (\beta \vee \theta )$.\end{proof}

\begin{proposition} For any $\alpha ,\beta ,\theta \in {\rm Con}(A)$:\begin{enumerate}
\item\label{impliquo1} $(\alpha \vee \theta )/\theta \rightarrow (\beta \vee \theta )/\theta =((\alpha \vee \theta )\rightarrow (\beta \vee \theta ))/\theta =(\alpha \rightarrow (\beta \vee \theta ))/\theta $;
\item\label{impliquo2} $((\alpha \vee \theta )/\theta )^{\perp }=(\alpha \rightarrow \theta )/\theta $.\end{enumerate}\label{impliquo}\end{proposition}

\begin{proof} By Lemma \ref{implic}, (\ref{implic1}), $\alpha \rightarrow (\beta \vee \theta )\supseteq \beta \vee \theta \supseteq \theta $ and $(\alpha \vee \theta )\rightarrow (\beta \vee \theta )\supseteq \beta \vee \theta \supseteq \theta $.

\noindent (\ref{impliquo1}) For any $\gamma \in [\theta )$, we have: $\gamma /\theta \subseteq (\alpha \vee \theta )/\theta \rightarrow (\beta \vee \theta )/\theta $ iff $[\gamma /\theta ,(\alpha \vee \theta )/\theta ]_{A/\theta }\subseteq (\beta \vee \theta )/\theta $ iff $([\gamma ,\alpha \vee \theta ]_A\vee \theta )/\theta \subseteq (\beta \vee \theta )/\theta $ iff $[\gamma ,\alpha \vee \theta ]_A\vee \theta \subseteq \beta \vee \theta $ iff $[\gamma ,\alpha \vee \theta ]_A\subseteq \beta \vee \theta $ iff $\gamma \subseteq (\alpha \vee \theta )\rightarrow (\beta \vee \theta )$ iff $\gamma /\theta \subseteq ((\alpha \vee \theta )\rightarrow (\beta \vee \theta ))/\theta $. Since ${\rm Con}(A/\theta )=\{\zeta /\theta \ |\ \zeta \in [\theta )\}$, by taking $\gamma /\theta =(\alpha \vee \theta )/\theta \rightarrow (\beta \vee \theta )/\theta $ and then $\gamma /\theta =(\alpha \rightarrow \beta )/\theta $ in the equivalences above we obtain the first equality in the enunciation; the second follows from Lemma \ref{implic}, (\ref{implic2}).

\noindent (\ref{impliquo2}) Take $\beta =\Delta _A$ in (\ref{impliquo1}).\end{proof}

\begin{lemma} Let $\alpha ,\beta \in {\rm Con}(A)$. Then:\begin{enumerate}
\item\label{proprperp0} $\Delta _A^{\perp }=\nabla _A$ and, if $[\theta ,\nabla _A]_A=\theta $ for all $\theta \in {\rm Con}(A)$, in particular if $\V $ is congruence--modular and semi--degenerate, then $\nabla _A^{\perp }=\Delta _A$;
\item\label{proprperp1} $\alpha \subseteq \beta $ implies $\beta ^{\perp }\subseteq \alpha ^{\perp }$, and: $\beta ^{\perp }\subseteq \alpha ^{\perp }$ iff $\alpha ^{\perp \perp }\subseteq \beta ^{\perp \perp }$, in particular $\alpha ^{\perp }=\beta ^{\perp }$ iff $\alpha ^{\perp \perp }=\beta ^{\perp \perp }$;
\item\label{proprperp2} $\alpha \subseteq \alpha ^{\perp \perp }$ and $\alpha ^{\perp \perp \perp }=\alpha ^{\perp }$;
\item\label{proprperp3} $(\alpha \vee \beta )^{\perp }=\alpha ^{\perp }\cap \beta ^{\perp }=(\alpha ^{\perp }\cap \beta ^{\perp })^{\perp \perp }$;
\item\label{proprperp4} if $A$ is semiprime, then $[\alpha ,\beta ]_A^{\perp }=(\alpha \cap \beta )^{\perp }$ and $(\alpha \cap \beta )^{\perp \perp }=\alpha ^{\perp \perp }\cap \beta ^{\perp \perp }$;
\item\label{proprperp5} if $A$ is semiprime, then: $\alpha ^{\perp }\subseteq \beta ^{\perp }$ iff $[\alpha ,\beta ]_A^{\perp }=\beta ^{\perp }$;
\item\label{proprperp6} if $A$ is semiprime, then: $\alpha \subseteq \alpha ^{\perp }$ iff $\alpha =\Delta _A$.

\end{enumerate}\label{proprperp}\end{lemma}

\begin{proof}(\ref{proprperp0}) $\Delta _A^{\perp }=\max \{\theta \in {\rm Con}(A)\ |\ [\theta ,\Delta _A]=\Delta _A\}=\max ({\rm Con}(A))=\nabla _A$. If $[\theta ,\nabla _A]_A=\theta $ for all $\theta \in {\rm Con}(A)$, then $\nabla _A^{\perp }=\max \{\theta \in {\rm Con}(A)\ |\ [\theta ,\nabla _A]=\Delta _A\}=\max \{\theta \in {\rm Con}(A)\ |\ \theta =\Delta _A\} =\max \{\Delta _A\}=\Delta _A$.

\noindent (\ref{proprperp1}),(\ref{proprperp2}) If $\alpha \subseteq \beta $, then $\{\theta \in {\rm Con}(A)\ |\ [\alpha ,\theta ]_A=\Delta _A\}\supseteq \{\theta \in {\rm Con}(A)\ |\ [\beta ,\theta ]_A=\Delta _A\}$, hence $\beta ^{\perp }\subseteq \alpha ^{\perp }$, which thus, in turn, implies $\alpha ^{\perp \perp }\subseteq \beta ^{\perp \perp }$.

Since $[\alpha ,\alpha ^{\perp }]_A=\Delta _A$, it follows that $\alpha \subseteq \alpha ^{\perp \perp }$, hence $\alpha ^{\perp }\subseteq \alpha ^{\perp \perp \perp }$ if we replace $\alpha $ by $\alpha ^{\perp }$ in this inclusion, but also $\alpha ^{\perp \perp \perp }\subseteq \alpha ^{\perp }$ by the above, therefore $\alpha ^{\perp }=\alpha ^{\perp \perp \perp }$.

Hence $\alpha ^{\perp \perp }\subseteq \beta ^{\perp \perp }$ implies $\beta ^{\perp }=\beta ^{\perp \perp \perp }\subseteq \alpha ^{\perp \perp \perp }=\alpha ^{\perp }$.

\noindent (\ref{proprperp3}) For any $\theta \in {\rm Con}(A)$, we have: $[\theta ,\alpha ]_A=[\theta ,\beta ]_A=\Delta _A$ iff $[\theta ,\alpha \vee \beta ]_A=\Delta _A$, hence: $\theta \subseteq \alpha ^{\perp }\cap \beta ^{\perp }$ iff $\theta \subseteq (\alpha \vee \beta )^{\perp }$, thus: $\alpha ^{\perp }\cap \beta ^{\perp }=(\alpha \vee \beta )^{\perp }$. By (\ref{proprperp2}), $(\alpha \vee \beta )^{\perp }=(\alpha \vee \beta )^{\perp \perp \perp }=(\alpha ^{\perp }\cap \beta ^{\perp })^{\perp \perp }$.

\noindent (\ref{proprperp4}) If $A$ is semiprime, then, for any $\theta ,\zeta \in {\rm Con}(A)$, we have: $\theta \subseteq \zeta ^{\perp }$ iff $[\theta ,\zeta ]_A=\Delta _A$ iff $\theta \cap \zeta =\Delta _A$.

Hence, for any $\theta \in {\rm Con}(A)$: $\theta \subseteq [\alpha ,\beta ]_A^{\perp }$ iff $[\theta ,[\alpha ,\beta ]_A]_A=\Delta _A$ iff $\theta \cap \alpha \cap \beta =\Delta _A$ iff $[\theta ,\alpha \cap \beta ]_A=\Delta _A$ iff $\theta \subseteq (\alpha \cap \beta )^{\perp }$. By taking $\theta =[\alpha ,\beta ]_A^{\perp }$ and then $\theta =(\alpha \cap \beta )^{\perp }$ in the previous equivalences, we obtain: $[\alpha ,\beta ]_A^{\perp }=(\alpha \cap \beta )^{\perp }$.

If we denote by $\gamma =\alpha ^{\perp \perp }\cap \beta ^{\perp \perp }$ and by $\delta =(\alpha \cap \beta )^{\perp }=[\alpha ,\beta ]_A^{\perp }$ by the above, then:

$\gamma \subseteq \alpha ^{\perp \perp }$ and $\gamma \subseteq \beta ^{\perp \perp }$, thus $[\gamma ,\alpha ^{\perp }]_A=\Delta _A$ and $[\gamma ,\beta ^{\perp }]_A=\Delta _A$;

$[\delta ,\alpha \cap \beta ]_A=\Delta _A$, so $\delta \cap \alpha \cap \beta =\Delta _A$, thus $[\alpha \cap \delta ,\beta ]_A=\Delta _A$, hence $\alpha \cap \delta \subseteq \beta ^{\perp }$;

therefore $[\gamma ,\alpha \cap \delta ]_A=\Delta _A$, so $\gamma \cap \alpha \cap \delta =\Delta _A$, thus $[\gamma \cap \delta ,\alpha ]_A=\Delta _A$, so $\gamma \cap \delta \subseteq \alpha ^{\perp }$;

hence $[\gamma ,\gamma \cap \delta ]_A=\Delta _A$, so $\gamma \cap \delta =\gamma \cap \gamma \cap \delta =\Delta _A$, thus $[\gamma ,\delta ]_A=\Delta _A$, hence $\alpha ^{\perp \perp }\cap \beta ^{\perp \perp }=\gamma \subseteq \delta ^{\perp }=[\alpha ,\beta ]_A^{\perp \perp }=(\alpha \cap \beta )^{\perp \perp }$ by the above.

But $(\alpha \cap \beta )^{\perp \perp }\subseteq \alpha ^{\perp \perp }\cap \beta ^{\perp \perp }$ by (\ref{proprperp1}). Therefore $(\alpha \cap \beta )^{\perp \perp }=\alpha ^{\perp \perp }\cap \beta ^{\perp \perp }$.

\noindent (\ref{proprperp5}) By (\ref{proprperp4}), $[\alpha ,\beta ]_A^{\perp \perp }=(\alpha \cap \beta )^{\perp \perp }=\alpha ^{\perp \perp }\cap \beta ^{\perp \perp }$, thus, according to (\ref{proprperp1}) and (\ref{proprperp2}): $\alpha ^{\perp }\subseteq \beta ^{\perp }$ iff $\beta ^{\perp \perp }\subseteq \alpha ^{\perp \perp }$ iff $\alpha ^{\perp \perp }\cap \beta ^{\perp \perp }=\beta ^{\perp \perp }$ iff $[\alpha ,\beta ]_A^{\perp \perp }=\beta ^{\perp \perp }$ iff $[\alpha ,\beta ]_A^{\perp }=\beta ^{\perp }$.

\noindent (\ref{proprperp6}) If $A$ is semiprime, then: $\alpha \subseteq \alpha ^{\perp }$ iff $[\alpha ,\alpha ]_A=\Delta _A$ iff $\alpha \cap \alpha =\Delta _A$ iff $\alpha =\Delta _A$.\end{proof}

\begin{lemma}{\rm \cite[Proposition $4$, $(1)$]{agl}, \cite[Proposition $18$, Corollary $2$]{retic}} For any $\theta \in {\rm Con}(A)$:\begin{itemize}
\item $\rho _A(\theta )=\max \{\alpha \in {\rm Con}(A)\ |\ (\exists \, n\in \N ^*)\, ([\alpha ,\alpha ]_A^n\subseteq \theta )\}=\bigvee \{\alpha \in {\rm Con}(A)\ |\ (\exists \, n\in \N ^*)\, ([\alpha ,\alpha ]_A^n\subseteq \theta )\}=\bigvee \{\alpha \in {\cal K}(A)\ |\ (\exists \, n\in \N ^*)\, ([\alpha ,\alpha ]_A^n\subseteq \theta )\}=\bigvee \{\alpha \in {\rm PCon}(A)\ |\ (\exists \, n\in \N ^*)\, ([\alpha ,\alpha ]_A^n\subseteq \theta )\}=\{(a,b)\in A^2\ |\ (\exists \, n\in \N ^*)\, ([Cg_A(a,b),Cg_A(a,b)]_A^n\subseteq \theta )\}$;
\item for any $\alpha \in {\rm Con}(A)$, $\alpha \subseteq \rho _A(\theta )$ iff there exists an $n\in \N ^*$ such that $[\alpha ,\alpha ]_A^n\subseteq \theta $;
\item $A$ is semiprime iff, for any $\alpha \in {\rm PCon}(A)$ and any $n\in \N ^*$, $[\alpha ,\alpha ]_A^n=\Delta _A$ implies $\alpha =\Delta _A$.\end{itemize}\label{charsemiprime}\end{lemma}

\begin{proposition} If $[\theta ,\nabla _A]_A=\theta $ for all $\theta \in {\rm Con}(A)$, in particular if $\V $ is either congruence--distributive or both congruence--modular and semi--degenerate, then: $A/\theta ^{\perp }$ is semiprime for all $\theta \in {\rm Con}(A)$ iff $A$ is semiprime.\label{perpsemiprime}\end{proposition}

\begin{proof} By \cite[Proposition~$5.22$]{fctretic} and Lemma \ref{negrad}, if $A$ is semiprime, then $A/\zeta $ is semiprime for all $\zeta \in {\rm RCon}(A)$, in particular $A/\theta ^{\perp }$ is semiprime for all $\theta \in {\rm Con}(A)$.

Conversely, if $A/\theta ^{\perp }$ is semiprime for all $\theta \in {\rm Con}(A)$, then $A/\nabla _A^{\perp }$ is semiprime; but $\nabla _A^{\perp }=\Delta _A$ by Lemma \ref{proprperp}, (\ref{proprperp0}), and $A/\nabla _A^{\perp }=A/\Delta _A\cong A$, thus $A$ is semiprime.\end{proof}

See also \cite{retic} for the following properties. By \cite[Proposition $6.7$]{stcommlat}, if $[\theta ,\nabla _A]_A=\theta $ for all $\theta \in {\rm Con}(A)$, in particular if $\V $ is either congruence--distributive or both congruence--modular and semi--degenerate, then:\begin{itemize}
\item for any $\varepsilon \in {\cal B}({\rm Con}(A))$ and any $\alpha \in {\rm Con}(A)$, $[\varepsilon ,\alpha ]_A=\varepsilon \cap \alpha $;
\item ${\cal B}({\rm Con}(A))$ is a Boolean sublattice of ${\rm Con}(A)$ whose complementation is $\cdot ^{\perp }$ and in which, by the above, the commutator equals the intersection.\end{itemize}

By \cite[Proposition $6.11$]{stcommlat}, if $\nabla _A\in {\cal K}(A)$ and $[\theta ,\nabla _A]_A=\theta $ for all $\theta \in {\rm Con}(A)$, in particular if $\V $ is congruence--modular and semi--degenerate, then ${\cal B}({\rm Con}(A))\subseteq {\cal K}(A)$.

Let us also note that, if the commutator of $A$ equals the intersection, in particular if $\V $ is congruence--distributive, then ${\rm Con}(A)$ is a frame, hence ${\cal B}({\rm Con}(A))$ is a complete Boolean sublattice of ${\rm Con}(A)$.

Following \cite{bur}, we say that an algebra $A$ is {\em hyperarchimedean} iff, for all $\theta \in {\rm PCon}(A)$, there exists an $n\in \N ^*$ such that $[\theta ,\theta ]_A^n\in {\cal B}({\rm Con}(A))$.

By the above, if the commutator of $A$ equals the intersection, in particular if $\V $ is congruence--distributive, then $A$ is hyperarchimedean iff ${\rm PCon}(A)\subseteq {\cal B}({\rm Con}(A))$ iff ${\rm Con}(A)\subseteq {\cal B}({\rm Con}(A))$ iff ${\cal B}({\rm Con}(A))={\rm Con}(A)$; furthermore, if the commutator of $A$ equals the intersection and $\nabla _A\in {\cal K}(A)$, in particular if $\V $ is congruence--distributive and semi--degenerate, then $A$ is hyperarchimedean iff ${\cal B}({\rm Con}(A))={\cal K}(A)={\rm Con}(A)$. Thus the hyperarchimedean members of a congruence--distributive variety are those with Boolean lattices of congruences and, if the variety is also semi--degenerate, then all congruences of its hyperarchimedean members are compact.

Extending the terminology used for rings in \cite{stcommlat}, we call $A$ a {\em strongly Baer}, respectively {\em Baer algebra} iff, for all $\theta \in {\rm Con}(A)$, respectively  all $\theta \in {\rm PCon}(A)$, we have $\theta ^{\perp }\in {\cal B}({\rm Con}(A))$, iff the commutator lattice $({\rm Con}(A),[\cdot ,\cdot ]_A)$ is strongly Stone, respectively Stone.

\begin{lemma} If $[\theta ,\nabla _A]_A=\theta $ for all $\theta \in {\rm Con}(A)$, in particular if $\V $ is either congruence--distributive or both congruence--modular and semi--degenerate, then: $A$ is Baer iff, for all $\theta \in {\cal K}(A)$, we have $\theta ^{\perp }\in {\cal B}({\rm Con}(A))$.\label{charbaer}\end{lemma}

\begin{proof} The converse implication is trivial.

If $A$ is Baer and $\theta \in {\cal K}(A)$, so that $\displaystyle \theta =\bigvee _{i=1}^n\theta _i$ for some $n\in \N ^*$ and $\theta _1,\ldots ,\theta _n\in {\rm PCon}(A)$, then $\theta _1^{\perp },\ldots ,\theta _n^{\perp }\in {\cal B}({\rm Con}(A))$, so that, according to Lemma \ref{proprperp}, (\ref{proprperp3}), $\theta ^{\perp }=(\theta _1\vee \ldots \vee \theta _n)^{\perp }=\theta _1^{\perp }\cap \ldots \cap \theta _n^{\perp }\in {\cal B}({\rm Con}(A))$.\end{proof}

\begin{proposition} If $[\theta ,\nabla _A]_A=\theta $ for all $\theta \in {\rm Con}(A)$, in particular if $\V $ is either congruence--distributive or both congruence--modular and semi--degenerate, then:\begin{enumerate}
\item\label{baerhyp1} if $A$ is hyperarchimedean, then $A$ is strongly Baer;
\item\label{baerhyp2} if $A$ is strongly Baer, then $A$ is semiprime;
\item\label{baerhyp3} if $A$ is Baer and has principal commutators, then $A$ is semiprime.\end{enumerate}\label{baerhyp}\end{proposition}

\begin{proof} (\ref{baerhyp1}) By the above, if $A$ is hyperarchimedean, then ${\cal B}({\rm Con}(A))={\rm Con}(A)$, thus $A$ is strongly Baer.

\noindent (\ref{baerhyp2}) Assume that $A$ is strongly Baer and let $\theta \in {\rm Con}(A)$ such that $[\theta ,\theta ]_A^n=\Delta _A$ for some $n\in \N ^*$. By the basic properties of the implication above and the fact that, if $n\geq 2$, then $[\theta ,\theta ]_A^n=[[\theta ,\theta ]_A^{n-1},[\theta ,\theta ]_A^{n-1}]_A$, this implies that $[\theta ,\theta ]_A^{n-1}\subseteq ([\theta ,\theta ]_A^{n-1})^{\perp }$. But, since $A$ is strongly Baer, $([\theta ,\theta ]_A^{n-1})^{\perp }\in {\cal B}({\rm Con}(A))$ thus its commutator with any congruence of $A$ equals the intersection, hence $[\theta ,\theta ]_A^{n-1}=[\theta ,\theta ]_A^{n-1}\cap ([\theta ,\theta ]_A^{n-1})^{\perp }=[[\theta ,\theta ]_A^{n-1},([\theta ,\theta ]_A^{n-1})^{\perp }]_A=\Delta _A$. By turning the above into a recursive argument we get that $[\theta ,\theta ]_A=\Delta _A$ and then that $\theta =\Delta _A$. By Lemma \ref{charsemiprime}, it follows that $A$ is semiprime.

\noindent (\ref{baerhyp3}) By an analogous argument to that of (\ref{baerhyp2}), taking $\theta \in {\rm PCon}(A)$, so that $[\theta ,\theta ]_A^n\in {\rm PCon}(A)$ for any $n\in \N ^*$ since $A$ has principal commutators.\end{proof}

\section{The Minimal Prime Spectrum}
\label{minspec}

\h Throughout this section, we will assume that the commutator of $A$ is commutative and distributive w.r.t. arbitrary joins, which holds if $\V $ is congruence--modular.

By an argument based on Zorn's Lemma, it follows that:\begin{itemize}
\item any prime congruence of $A$ includes a minimal prime congruence, hence $\rho _A(\Delta _A)=\bigcap {\rm Spec}(A)=\bigcap {\rm Min}(A)$;
\item moreover, for any $\theta \in {\rm Con}(A)$ and any $\psi \in V_A(\theta )=[\theta )\cap {\rm Spec}(A)$, there exists a $\phi \in Min(V_A(\theta ))=Min([\theta )\cap {\rm Spec}(A))$ such that $\phi \subseteq \psi $, hence:\end{itemize}

\begin{remark} For any $\theta \in {\rm Con}(A)$, we have:\begin{itemize}
\item $\rho _A(\theta )=\bigcap Min(V_A(\theta ))=\bigcap Min([\theta )\cap {\rm Spec}(A))$;
\item $D_A(\theta )\cap {\rm Min}(A)=\emptyset $ iff $V_A(\theta )\cap {\rm Min}(A)={\rm Min}(A)$ iff $[\theta )\cap {\rm Min}(A)={\rm Min}(A)$ iff ${\rm Min}(A)\subseteq [\theta )$ iff $\theta \subseteq \bigcap {\rm Min}(A)$ iff $\theta \subseteq \rho _A(\Delta _A)$ iff $\rho _A(\theta )=\rho _A(\Delta _A)$.
\item $D_A(\theta )\cap {\rm Min}(A)={\rm Min}(A)$ iff $V_A(\theta )\cap {\rm Min}(A)=\emptyset $;

$V_A(\theta )=\emptyset $ iff $\rho_A(\theta )=\nabla _A$, which holds if $\theta =\nabla _A$; recall from \cite{retic} that, if $\nabla _A\in {\cal K}(A)$ and $[\nabla _A,\nabla _A]_A=\nabla _A$, then $\nabla _A/\!\equiv _A=\{\nabla _A\}$, so: $\rho_A(\theta )=\nabla _A$ iff $\theta =\nabla _A$;

clearly, $V_A(\theta )=\emptyset $ implies $V_A(\theta )\cap {\rm Min}(A)=\emptyset $; the converse implication holds iff ${\rm Min}(A)={\rm Spec}(A)$ iff ${\rm Spec}(A)$ is an antichain.\end{itemize}

Indeed, ${\rm Spec}(A)$ is an antichain iff ${\rm Min}(A)={\rm Spec}(A)$, case in which $V_A(\theta )=V_A(\theta )\cap {\rm Min}(A)$.

Now, if $V_A(\theta )\cap {\rm Min}(A)=\emptyset $ implies $V_A(\theta )=\emptyset $, then let us assume by absurdum that ${\rm Min}(A)\neq {\rm Spec}(A)$, that is ${\rm Spec}(A)\nsubseteq {\rm Min}(A)$, so that there exists $\phi \in {\rm Spec}(A)\setminus {\rm Min}(A)$. But then $V_A(\phi )\cap {\rm Min}(A)=\emptyset $, while $V_A(\phi )\neq \emptyset $ since $\phi \in V_A(\phi )$; a contradiction.\label{rhomin}\end{remark}

\begin{proposition} If $\nabla _A\in {\cal K}(A)$, in particular if $\V $ is congruence--modular and semi--degenerate, then, for any $\theta \in {\rm Con}(A)$ and any $\phi \in V_A(\theta )$, the following are equivalent:\begin{enumerate}
\item\label{charminoverth1} $\phi \in Min(V_A(\theta ))$;
\item\label{charminoverth2} $\nabla _A\setminus \phi $ is a maximal element of the set of $m$--systems of $A$ which are disjoint from $\theta $.\end{enumerate}\label{charminoverth}\end{proposition}

\begin{proof} By Remark \ref{msist}, $\nabla _A\setminus \phi $ is an $m$--system, which is, of course, disjoint from $\theta $ since $(\nabla _A\setminus \phi )\cap \theta \subseteq (\nabla _A\setminus \phi )\cap \phi =\emptyset $.

\noindent (\ref{charminoverth1})$\Rightarrow $(\ref{charminoverth2}): By an application of Zorn's Lemma, it follows that there exists a maximal element $M$ of the set of $m$--systems of $A$ which include $\nabla _A\setminus \phi $ and are disjoint from $\theta $, so that $\nabla _A\setminus \phi \subseteq M\subseteq \nabla _A\setminus \theta $ and, furthermore, $M$ is a maximal element of the set of $m$--systems of $A$ which are disjoint from $\theta $.

According to Lemma \ref{msistspec}, there exists a $\psi \in {\rm Max}\{\alpha \in {\rm Con}(A)\ |\ \theta \subseteq \alpha ,M\cap \alpha =\emptyset \}\subseteq {\rm Spec}(A)$, so that $\psi \in V_A(\theta )$ and $(\nabla _A\setminus \phi )\cap \psi \subseteq M\cap \psi =\emptyset $, thus $\nabla _A\setminus \phi \subseteq M\subseteq \nabla _A\setminus \psi $, hence $\psi \subseteq \phi $.

Since $\phi \in Min(V_A(\theta ))$, it follows that $\phi =\psi $, thus $\nabla _A\setminus \phi =M$, which is a maximal element of the set of $m$--systems of $A$ which are disjoint from $\theta $. 

\noindent (\ref{charminoverth2})$\Rightarrow $(\ref{charminoverth1}): Let $\mu $ be a minimal element of $V_A(\theta )$ with $\mu \subseteq \phi $.

By Remark \ref{msist}, $\nabla _A\setminus \mu $ is an $m$--system, which is disjoint from $\theta $ since $(\nabla _A\setminus \mu )\cap \theta \subseteq (\nabla _A\setminus \mu )\cap \mu =\emptyset $, and $\nabla _A\setminus \phi \subseteq \nabla _A\setminus \mu $.

Since $\nabla _A\setminus \phi $ is a maximal element of the set of $m$--systems of $A$ which are disjoint from $\theta $, it follows that $\nabla _A\setminus \phi =\nabla _A\setminus \mu $, thus $\phi =\mu $, which is a minimal element of $V_A(\theta )$.\end{proof}

\begin{corollary} If $\nabla _A\in {\cal K}(A)$, in particular if $\V $ is congruence--modular and semi--degenerate, then, for any $\phi \in {\rm Spec}(A)$, the following are equivalent:\begin{itemize}
\item $\phi \in {\rm Min}(A)$;
\item $\nabla _A\setminus \phi $ is a maximal element of the set of $m$--systems of $A$ which are disjoint from $\Delta _A$.\end{itemize}\label{charmin}\end{corollary}

\begin{proof} By Proposition \ref{charminoverth} for $\theta =\Delta _A$.\end{proof}

\begin{lemma}{\rm \cite{sim}} If $L$ is a bounded distributive lattice and $P\in {\rm Spec}_{\rm Id}(L)$, then the following are equivalent:\begin{itemize}
\item $P\in {\rm Min}_{\rm Id}(L)$;
\item for any $x\in P$, ${\rm Ann}_L(x)\nsubseteq P$.\end{itemize}\label{charminid}\end{lemma}

Recall from Section \ref{preliminaries} that ${\rm Spec}(A)={\rm Mi}({\rm Con}(A))\cap {\rm RCon}(A)$. By \cite[Proposition $4.4$]{stcommlat}, if $A$ is semiprime, then all annihilators in $({\rm Con}(A),[\cdot ,\cdot ]_A)$ are lattice ideals of ${\rm Con}(A)$. 

Remember that, in the commutator lattice $({\rm Con}(A),[\cdot ,\cdot ]_A)$, $R({\rm Con}(A))={\rm RCon}(A)$ and ${\rm Spec}_{{\rm Con}(A)}={\rm Spec}(A)$, and that, since ${\rm Con}(A)/\!\equiv _A$ is a frame, the elements of ${\rm Spec}_{{\rm Con}(A)/\equiv _A}$ are exactly the meet--prime elements of ${\rm Con}(A)/\!\equiv _A$, thus, by the distributivity of ${\rm Con}(A)/\!\equiv _A$, ${\rm Spec}_{{\rm Con}(A)/\equiv _A}={\rm Mi}({\rm Con}(A)/\!\equiv _A)$.

\begin{lemma} If $A$ is semiprime, then:\begin{enumerate}
\item\label{radeq1} for any $U\subseteq {\rm Con}(A)$, ${\rm Ann}_{{\rm Con}(A)/\equiv _A}(U/\!\equiv _A)={\rm Ann}_{({\rm Con}(A),[\cdot ,\cdot ]_A)}(U)/\!\equiv _A$;
\item\label{radeq2} ${\rm Spec}_{{\rm Con}(A)/\equiv _A}=\{\phi /\!\equiv _A\ |\ \phi \in {\rm Spec}(A)\}$;
\item\label{radeq3} for all $\theta \in {\rm RCon}(A)$, $\theta /\!\equiv _A\cap {\rm RCon}(A)=\{\theta \}$ and $\theta =\max (\theta /\!\equiv _A)$;
\item\label{radeq4} $\phi \mapsto \phi /\!\equiv _A$ is an order isomorphism from ${\rm Spec}(A)$ to ${\rm Spec}_{{\rm Con}(A)/\equiv _A}$;
\item\label{radeq5} $R({\rm Con}(A)/\!\equiv _A)=\{\phi /\!\equiv _A\ |\ \phi \in {\rm RCon}(A)\}$; moreover, for any $\phi \in {\rm Con}(A)$, we have: $\phi \in {\rm RCon}(A)$ iff $\phi /\!\equiv _A\in R({\rm Con}(A)/\!\equiv _A)$; thus $\phi \mapsto \phi /\!\equiv _A$ is an order isomorphism from ${\rm RCon}(A)$ to $R({\rm Con}(A)/\!\equiv _A)$.\end{enumerate}\label{radeq}\end{lemma}

\begin{proof} (\ref{radeq1}) By \cite[Lemma $4.2$]{stcommlat}.

\noindent (\ref{radeq2}) By \cite[Proposition $6.2$]{stcommlat}.

\noindent (\ref{radeq3}) By \cite[Remark $5.11$]{stcommlat}.

\noindent (\ref{radeq4}) By (\ref{radeq2}), (\ref{radeq3}) and the fact that ${\rm Spec}(A)\subseteq {\rm RCon}(A)$ and ${\rm Spec}_{{\rm Con}(A)/\equiv _A}\subseteq R({\rm Con}(A)/\!\equiv _A)$.

\noindent (\ref{radeq5}) The equality follows from (\ref{radeq2}) and the definition of radical elements; by (\ref{radeq3}), we also obtain the equivalence and the order isomorphism.\end{proof}

\begin{remark} For any $\alpha ,\beta \in {\rm Con}(A)$, we have: $\alpha /\!\equiv _A\leq \beta/\!\equiv _A$ iff $\rho _A(\alpha )\subseteq \rho _A(\beta )$, because: $\alpha /\!\equiv _A\leq \beta/\!\equiv _A$ iff $\alpha /\!\equiv _A\wedge \beta/\!\equiv _A=\alpha /\!\equiv _A$ iff $(\alpha \cap \beta )/\!\equiv _A=\alpha /\!\equiv _A$ iff $\rho _A(\alpha \cap \beta )=\rho _A(\alpha )$ iff $\rho _A(\alpha )\cap \rho _A(\beta )=\rho _A(\alpha )$ iff $\rho _A(\alpha )\subseteq \rho _A(\beta )$.\end{remark}

In many of the following results, we will refer to these hypotheses:\vspace*{3pt}

\noindent \hyplA\ \ $\nabla _A\in {\cal K}(A)$ and ${\cal K}(A)$ is closed w.r.t. the commutator of $A$;\vspace*{3pt}

\noindent \hypmuA\ \ all principal ideals of ${\rm Con}(A)/\!\equiv _A$ generated by minimal prime elements are minimal prime ideals, that is: for any $p\in {\rm Min}_{{\rm Con}(A)/\equiv _A}$, we have $(p]\in {\rm Min}_{\rm Id}({\rm Con}(A)/\!\equiv _A)$.\vspace*{3pt}

Since an element of a lattice is prime iff the principal ideal it generates is prime, we have that, whenever a principal ideal of a lattice is a minimal prime ideal, it follows that its generator is a minimal prime element of that lattice. Hence condition \hypmuA\ is equivalent to:

\noindent $\bullet $ for any $p\in {\rm Con}(A)/\!\equiv _A$, we have: $p\in {\rm Min}_{{\rm Con}(A)/\equiv _A}$ iff $(p]\in {\rm Min}_{\rm Id}({\rm Con}(A)/\!\equiv _A)$.

Recall that ${\cal K}(A)={\rm Cp}({\rm Con}(A))$. Note that \hypmuA\ holds if all prime ideals of ${\rm Con}(A)/\!\equiv _A$ are principal, in particular if all ideals of ${\rm Con}(A)/\!\equiv _A$ are principal, that is if ${\rm Con}(A)/\!\equiv _A$ is compact, in particular if ${\rm Con}(A)/\!\equiv _A={\cal K}(A)/\!\equiv _A$, that is if ${\rm Con}(A)/\!\equiv _A={\cal L}(A)$ in the case when ${\cal K}(A)$ is closed w.r.t. the commutator, in particular if ${\rm Con}(A)$ is compact, that is if ${\rm Con}(A)={\cal K}(A)$.

If ${\rm Con}(A)={\cal K}(A)$, then \hyplA\ is trivially satisfied.

Recall from \cite{retic}, that, if \hyplA\ is satisfied, then the {\em reticulation} ${\cal L}(A)$ of $A$ can be constructed as the lattice ${\cal L}(A)={\cal K}(A)/\!\equiv _A$, which is a bounded sublattice of the frame ${\rm Con}(A)/\!\equiv _A$ and thus ${\cal L}(A)$ is a bounded distributive lattice.

Recall, also, that, if ${\cal V}$ is congruence--modular and semi--degenerate, then $\nabla _A\in {\cal K}(A)$.

\begin{proposition} Assume that $A$ is semiprime and let $\phi \in {\rm Spec}(A)$. Let us consider the following statements:\begin{enumerate}
\item\label{1phiMin} $\phi \in {\rm Min}(A)$;
\item\label{2K} for any $\alpha \in {\cal K}(A)$, $\alpha \subseteq \phi $ implies $\alpha ^{\perp }\nsubseteq \phi $;
\item\label{3K} for any $\alpha \in {\cal K}(A)$, $\alpha \subseteq \phi $ iff $\alpha ^{\perp }\nsubseteq \phi $;
\item\label{2Con} for any $\alpha \in {\rm Con}(A)$, $\alpha \subseteq \phi $ implies $\alpha ^{\perp }\nsubseteq \phi $;
\item\label{3Con} for any $\alpha \in {\rm Con}(A)$, $\alpha \subseteq \phi $ iff $\alpha ^{\perp }\nsubseteq \phi $.\end{enumerate}

If \hyplA\ holds, then statements (\ref{1phiMin}), (\ref{2K}) and (\ref{3K}) are equivalent.

If \hypmuA\ holds, then statements (\ref{1phiMin}), (\ref{2Con}) and (\ref{3Con}) are equivalent.\label{charmincg}\end{proposition}

\begin{proof} {\bf Case 1:} Assume that $A$ satisfies \hyplA .

\noindent (\ref{1phiMin})$\Leftrightarrow $(\ref{2K}): Recall from \cite{retic} that we have the following maps, which are clearly order--preserving \cite[Lemma~$11$.(i)]{retic}:

\noindent $\theta \mapsto \theta ^{\ast }$ from ${\rm Con}(A)$ to ${\rm Id}({\cal L}(A))$, defined by: $\theta ^{\ast }=((\theta ]\cap {\cal K}(A))/\!\equiv _A$ for all $\theta \in {\rm Con}(A)$;

\noindent $I\mapsto I_{\ast }$ from ${\rm Id}({\cal L}(A))$ to ${\rm Con}(A)$, defined by: $I_{\ast }=\bigvee \{\gamma \in {\cal K}(A)\ |\ \gamma /\!\equiv _A\in I\}$ for all $I\in {\rm Id}({\cal L}(A))$.

By \cite[Proposition~$11$]{retic}, these maps restrict to order isomorphisms between ${\rm Spec}(A)$ and ${\rm Spec}_{\rm Id}({\cal L}(A))$, inverses of each other, thus they further restrict to mutually inverse order isomorphisms between ${\rm Min}(A)$ and ${\rm Min}_{\rm Id}({\cal L}(A))$.

Let $\beta \in {\cal K}(A)$ and all $\psi \in {\rm Spec}(A)$, arbitrary. By the above, $(\psi ^{\ast })_{\ast }=\psi $. Since $\beta \in {\cal K}(A)$, $(\beta /\!\equiv _A]_{{\cal L}(A)}=(\beta /\!\equiv _A]_{{\rm Con}(A)/\equiv _A}\cap {\cal L}(A)=(\beta ]_{{\rm Con}(A)}/\!\equiv _A\cap {\cal K}(A)/\!\equiv _A=((\beta ]_{{\rm Con}(A)}\cap {\cal K}(A))/\!\equiv _A=\beta ^{\ast }$, hence ${\rm Ann}_{{\cal L}(A)}(\beta /\!\equiv _A)={\rm Ann}_{{\cal L}(A)}((\beta /\!\equiv _A]_{{\cal L}(A)})={\rm Ann}_{{\cal L}(A)}(\beta ^{\ast })$. By \cite[Lemma~$27$]{retic}, since $A$ is semiprime, we have: ${\rm Ann}_{{\cal L}(A)}(\beta ^{\ast })\subseteq \psi ^{\ast }$ iff $\beta ^{\perp }\subseteq (\psi ^{\ast })_{\ast }$, that is $\beta ^{\perp }\subseteq \psi $.

Hence: $\phi \in {\rm Min}(A)$ iff $\phi ^{\ast }\in {\rm Min}_{\rm Id}({\cal L}(A))$. By Lemma \ref{charminid}, the latter is equivalent to: $(\forall \, x\in \phi ^{\ast })\, ({\rm Ann}_{{\cal L}(A)}(x)\nsubseteq \phi ^{\ast })$, that is: $(\forall \, \alpha \in (\phi ]\cap {\cal K}(A))\, ({\rm Ann}_{{\cal L}(A)}(\alpha /\!\equiv _A)\nsubseteq \phi ^{\ast })$, which means that: $(\forall \, \alpha \in (\phi ]\cap {\cal K}(A))\, ({\rm Ann}_{{\cal L}(A)}(\alpha ^{\ast })\nsubseteq \phi ^{\ast })$, which is equivalent to: $(\forall \, \alpha \in (\phi ]\cap {\cal K}(A))\, (\alpha ^{\perp }\nsubseteq \phi )$, that is: $(\forall \, \alpha \in {\cal K}(A))\, (\alpha \subseteq \phi \Rightarrow \alpha ^{\perp }\nsubseteq \phi )$.

\noindent (\ref{3K})$\Rightarrow $(\ref{2K}): Trivial.

\noindent (\ref{2K})$\Rightarrow $(\ref{3K}): If $\alpha \in {\cal K}(A)$ is such that $\alpha ^{\perp }\nsubseteq \phi $, then, since $[\alpha ,\alpha ^{\perp }]_A=\Delta _A\subseteq \phi \in {\rm Spec}(A)$, it follows that $\alpha \subseteq \phi $.

\noindent {\bf Case 2:} Now assume that $A$ satisfies \hypmuA .

\noindent (\ref{3Con})$\Rightarrow $(\ref{2Con}): Trivial.

\noindent (\ref{2Con})$\Rightarrow $(\ref{3Con}): Analogous to the proof of (\ref{2K})$\Rightarrow $(\ref{3K}).

\noindent (\ref{1phiMin})$\Leftrightarrow $(\ref{2Con}): By Lemma \ref{radeq}.(\ref{radeq4}), the condition that $\phi \in {\rm Spec}(A)$ is equivalent to $\phi /\!\equiv _A\in {\rm Spec}_{{\rm Con}(A)/\equiv _A}$, which is equivalent to $(\phi /\!\equiv _A]\in {\rm Spec}_{\rm Id}({\rm Con}(A)/\!\equiv _A)$. 

Again by Lemma \ref{radeq}.(\ref{radeq4}), $\phi \in {\rm Min}(A)$ iff $\phi /\!\equiv _A\in {\rm Min}_{{\rm Con}(A)/\equiv _A}$, which is equivalent to $(\phi /\!\equiv _A]\in {\rm Min}_{\rm Id}({\rm Con}(A)/\!\equiv _A)$. By Lemma \ref{charminid} and Lemma \ref{radeq}.(\ref{radeq1}), the latter is equivalent to the fact that, for any $\alpha \in (\phi ]$, $(\alpha ^{\perp }/\!\equiv _A]=(\alpha ^{\perp }]/\!\equiv _A={\rm Ann}_{({\rm Con}(A),[\cdot ,\cdot ]_A)}(\alpha )/\!\equiv _A={\rm Ann}_{{\rm Con}(A)/\equiv _A}(\alpha /\!\equiv _A)\nsubseteq (\phi /\!\equiv _A]=(\phi ]/\!\equiv _A$, that is $\alpha ^{\perp }/\!\equiv _A\notin (\phi /\!\equiv _A]=(\phi ]/\!\equiv _A$. Since $\phi \in {\rm Spec}(A)\subseteq {\rm RCon}(A)$ and thus $\phi =\max(\phi /\!\equiv _A)$ by Lemma \ref{radeq}.(\ref{radeq3}), this condition is equivalent to $\alpha ^{\perp }\notin (\phi ]$, that is $\alpha ^{\perp }\nsubseteq \phi $.\end{proof}

\begin{example} Note that the equivalence in Proposition \ref{charmincg} for the case when $A$ satisfies \hyplA\ does not hold for $\alpha \in {\rm Con}(A)$, arbitrary. Indeed, if we let $A$ be the Boolean subalgebra of the power set ${\cal P}(\N )$ of the set $\N $ of natural numbers formed of the finite and the cofinite subsets of $\N $: $A=\{S\ |\ S\subseteq \N ,|S|<\aleph_0\mbox{ or }|\N \setminus S|<\aleph_0\}$, then, since $A$ is a Boolean algebra, its lattice of congruences is isomorphic to its lattice of filters, and obviously this lattice isomorphism $\varphi :{\rm Filt}(A)\rightarrow {\rm Con}(A)$ takes the set ${\rm Spec}_{{\rm Filt}(A)}$ of the prime elements of the lattice ${\rm Filt}(A)$ of the filters of $A$, which equals the set ${\rm Spec}_{\rm Filt}(A)={\rm Max}_{\rm Filt}(A)$ of the prime and thus maximal filters of $A$ by a routine proof, to ${\rm Spec}_{\rm Con}(A)={\rm Spec}(A)={\rm Max}(A)={\rm Min}(A)$ since $A$ is a Boolean algebra, therefore ${\rm Min}_{\rm Filt}(A):=Min({\rm Spec}_{\rm Filt}(A))={\rm Spec}_{\rm Filt}(A)={\rm Max}_{\rm Filt}(A)=Max({\rm Filt}(A)\setminus \{\{\N \}\})$. Now let us consider the filter $P$ of $A$ formed of the cofinite subsets of $\N $: $P=\{S\ |\ S\subseteq \N ,|\N \setminus S|<\aleph_0\}$. It is well known that ${\rm Spec}_{\rm Filt}(A)=Max({\rm Filt}(A)\setminus \{\{\N \}\})=\{M\cap A\ |\ M\in Max({\cal P}(\N )\setminus \{\{\N \}\})\}\cup \{P\}=\{[\{a\})_{{\cal P}(\N )}\cap A\ |\ a\in M\}\cup \{P\}$, in particular $P$ is a prime and thus a minimal prime filter of $A$. $P$ is clearly not a principal, thus not a compact filter of $A$. Since Boolean algebras are congruence--distributive, the commutator $[\cdot ,\cdot ]_A$ of $A$ equals the intersection, thus the commutator lattice $({\rm Con}(A),[\cdot ,\cdot ]_A=\cap )$ is isomorphic to the commutator lattice $({\rm Filt}(A),\cap )$, also endowed the commutator operation equalling the intersection, in which $P^{\perp }=\max \{F\in {\rm Filt}(A)\ |\ P\cap F=\{\N \}\}=\max \{\{\N \}\}=\{\N \}$, since any nontrivial filter $F$ of $A$ contains a proper subset $S$ of $\N $, which must thus be such that an $a\in \N $ does not belong to $S$, hence $S$ is included in the proper cofinite subset $\N \setminus \{a\}$ of $\N $, so $\N \setminus \{a\}\in P\cap F$, which means that no nontrivial filter of $A$ satisfies $P\cap F=\{\N \}$. So $P^{\perp }=\{\N \}\subset P$, of course, $P\subseteq P$. Therefore $\varphi (P)\in {\rm Spec}(A)={\rm Min}(A)$ and $\varphi (P)\in {\rm Con}(A)\setminus {\rm Cp}({\rm Con}(A))={\rm Con}(A)\setminus {\cal K}(A)$, $\varphi (P)^{\perp }=\varphi (P^{\perp })=\varphi (\{\N \})=\Delta _A\subset \varphi (P)$ and $\varphi (P)\subseteq \varphi (P)$, hence $\varphi (P)\subseteq \varphi (P)$ does not imply $\varphi (P)^{\perp }\nsubseteq \varphi (P)$.\end{example}

\begin{corollary} Let $\phi \in {\rm Spec}(A)$ and let us consider the following statements:\begin{enumerate}
\item\label{1corphiMin} $\phi \in {\rm Min}(A)$;
\item\label{2corK} for any $\alpha \in {\cal K}(A)$, $\alpha \subseteq \phi $ implies $\alpha \rightarrow \rho _A(\Delta _A)\nsubseteq \phi $;
\item\label{2corCon} for any $\alpha \in {\rm Con}(A)$, $\alpha \subseteq \phi $ implies $\alpha \rightarrow \rho _A(\Delta _A)\nsubseteq \phi $.\end{enumerate}

If \hyplA\ holds, then (\ref{1corphiMin}) is equivalent to (\ref{2corK}).

If \hypmuA\ holds, then (\ref{1corphiMin}) is equivalent to (\ref{2corCon}).\label{quomincg}\end{corollary}

\begin{proof} {\bf Case 1:} Assume that $A$ satisfies \hyplA . Then we have the following equivalences.

$\phi \in {\rm Min}(A)$ iff $\phi /\rho _A(\Delta _A)\in {\rm Min}(A)$, which, by Proposition \ref{charmincg}, since $A/\rho _A(\Delta _A)$ is semiprime, is equivalent to the fact that, for any $\alpha \in {\cal K}(A)$, $(\alpha \vee \rho _A(\Delta _A))/\rho _A(\Delta _A)\subseteq \phi /\rho _A(\Delta _A)$ implies $((\alpha \vee \rho _A(\Delta _A))/\rho _A(\Delta _A))^{\perp }\nsubseteq \phi /\rho _A(\Delta _A)$, that is $\alpha \vee \rho _A(\Delta _A)\subseteq \phi $ implies $(\alpha \rightarrow \rho _A(\Delta _A))/\rho _A(\Delta _A)\nsubseteq \phi /\rho _A(\Delta _A)$ according to Proposition \ref{impliquo}, (\ref{impliquo2}), that is $\alpha \subseteq \phi $ implies $\alpha \rightarrow \rho _A(\Delta _A)\nsubseteq \phi $ since $\phi $ is prime and thus $\rho _A(\Delta _A)\subseteq \phi $.

\noindent {\bf Case 2:}  Assume that $A$ satisfies \hypmuA . Then the proof goes the same as above, but for all $\alpha \in {\rm Con}(A)$.\end{proof}

\section{Two Topologies on the Minimal Prime Spectrum}
\label{2topminspec}

\hnosp Throughout this section, we will assume that the commutator of $A$ is commutative and  distributive w.r.t. arbitrary joins, which holds in the particular case when $\V $ is congruence--modular.\vspace*{7pt}

Clearly, the Stone topology ${\cal S}_{\rm Spec}(A)$ of ${\rm Spec}(A)$ induces the topology ${\cal S}_{\rm Min}(A)=\{D_A(\theta )\cap {\rm Min}(A)\ |\ \theta \in {\rm Con}(A)\}$ on ${\rm Min}(A)$, which has $\{D_A(a,b)\cap {\rm Min}(A)\ |\ a,b\in A\}$ as a basis and $\{V_A(\theta )\cap {\rm Min}(A)\ |\ \theta \in {\rm Con}(A)\}$ as the family of closed sets. ${\cal S}_{\rm Min}(A)$ is called the {\em Stone} or {\em spectral topology} on ${\rm Min}(A)$.

\h Throughout the rest of this section, we will also assume that $A$ is semiprime.

\begin{lemma} $\theta ^{\perp }=\bigcap (V_A(\theta ^{\perp })\cap {\rm Min}(A))$ for every $\theta \in {\rm Con}(A)$.\label{negradthus}\end{lemma}

\begin{proof} Let $\theta \in {\rm Con}(A)$. Clearly, $\theta ^{\perp }\subseteq \bigcap (V_A(\theta ^{\perp })\cap {\rm Min}(A))$.

Let us denote by $\alpha =\bigcap (V_A(\theta ^{\perp })\cap {\rm Min}(A))$, so that $\alpha \subseteq \mu $ for any $\mu \in V_A(\theta ^{\perp })\cap {\rm Min}(A)$.

Assume by absurdum that $\alpha \nsubseteq \theta ^{\perp }$, so that $[\alpha ,\theta ]_A\neq \Delta _A=\rho _A(\Delta _A)=\bigcap {\rm Min}(A)$ since $A$ is semiprime, therefore $[\alpha ,\theta ]_A\nsubseteq \phi $ for some $\phi \in {\rm Min}(A)$, which implies that $\theta \nsubseteq \phi $ and $\alpha \nsubseteq \phi $, hence $\phi \notin V_A(\theta ^{\perp })$, that is $\theta ^{\perp }\nsubseteq \phi $. So $\theta \nsubseteq \phi $ and $\theta ^{\perp }\nsubseteq \phi $, while $[\theta ,\theta ^{\perp }]_A=\Delta _A\subseteq \phi $, which contradicts the fact that $\phi \in {\rm Min}(A)\subseteq {\rm Spec}(A)$.

Therefore $\bigcap (V_A(\theta ^{\perp })\cap {\rm Min}(A))=\alpha \subseteq \theta ^{\perp }$, hence the equality.\end{proof}

\begin{remark} By Lemma \ref{negradthus}, for any $\alpha ,\beta \in {\rm Con}(A)$, we have: $\alpha ^{\perp }=\beta ^{\perp }$ iff $V_A(\alpha ^{\perp })\cap {\rm Min}(A)=V_A(\beta ^{\perp })\cap {\rm Min}(A)$ iff $D_A(\alpha ^{\perp })\cap {\rm Min}(A)=D_A(\beta ^{\perp })\cap {\rm Min}(A)$.\label{eqneg}\end{remark}

\begin{proposition} For any $\alpha ,\beta ,\gamma \in {\rm Con}(A)$, we consider the following statements:\begin{enumerate}
\item\label{negmin1} $V_A(\alpha )\cap {\rm Min}(A)=V_A(\alpha ^{\perp \perp })\cap {\rm Min}(A)=D_A(\alpha ^{\perp })\cap {\rm Min}(A)$ and $D_A(\alpha )\cap {\rm Min}(A)=D_A(\alpha ^{\perp \perp })\cap {\rm Min}(A)=V_A(\alpha ^{\perp })\cap {\rm Min}(A)$;
\item\label{negmin2} $\alpha ^{\perp }\cap \beta ^{\perp }=\gamma ^{\perp }$ iff $V_A(\alpha )\cap V_A(\beta )\cap {\rm Min}(A)=V_A(\gamma )\cap {\rm Min}(A)$;
\item\label{negmin3} $\alpha ^{\perp \perp }=\beta ^{\perp }$ iff $\alpha ^{\perp }=\beta ^{\perp \perp }$ iff $V_A(\alpha )\cap {\rm Min}(A)=V_A(\beta ^{\perp })\cap {\rm Min}(A)$ iff $V_A(\alpha )\cap {\rm Min}(A)=D_A(\beta )\cap {\rm Min}(A)$ iff $D_A(\alpha ^{\perp })\cap {\rm Min}(A)=D_A(\beta )\cap {\rm Min}(A)$.\end{enumerate}

If $A$ satisfies \hyplA , then the statements above hold for all $\alpha ,\beta ,\gamma \in {\cal K}(A)$.

If $A$ satisfies \hypmuA , then the statements above hold for all $\alpha ,\beta ,\gamma \in {\rm Con}(A)$.\label{negmin}\end{proposition}

\begin{proof} Let $\phi \in {\rm Min}(A)$.

\noindent {\bf Case 1:} Assume that \hyplA\ holds and let $\alpha ,\beta ,\gamma \in {\cal K}(A)$.

\noindent (\ref{negmin1}) By Proposition \ref{charmincg}, $\phi \in V_A(\alpha )$ iff $\phi \in D_A(\alpha ^{\perp })$, hence also $\phi \notin V_A(\alpha )$ iff $\phi \notin D_A(\alpha ^{\perp })$, that is $\phi \in D_A(\alpha )$ iff $\phi \in V_A(\alpha ^{\perp })$. Therefore $V_A(\alpha )\cap {\rm Min}(A)=D_A(\alpha ^{\perp })\cap {\rm Min}(A)$ and $D_A(\alpha )\cap {\rm Min}(A)=V_A(\alpha ^{\perp })\cap {\rm Min}(A)$, hence also $V_A(\alpha ^{\perp \perp })\cap {\rm Min}(A)=D_A(\alpha ^{\perp \perp \perp })\cap {\rm Min}(A)=D_A(\alpha ^{\perp })\cap {\rm Min}(A)$ and $D_A(\alpha ^{\perp \perp })\cap {\rm Min}(A)=V_A(\alpha ^{\perp \perp \perp })\cap {\rm Min}(A)=V_A(\alpha ^{\perp })\cap {\rm Min}(A)$ by Lemma \ref{proprperp}, (\ref{proprperp2}).

\noindent (\ref{negmin2}) By (\ref{negmin1}), along with Proposition \ref{proprperp}, (\ref{proprperp3}), and Remark \ref{eqneg}, $\alpha ^{\perp }\cap \beta ^{\perp }=\gamma ^{\perp }$ iff $(\alpha \vee \beta )^{\perp }=\gamma ^{\perp }$ iff $V_A((\alpha \vee \beta )^{\perp })\cap {\rm Min}(A)=V_A(\gamma ^{\perp })\cap {\rm Min}(A)$ iff $(D_A(\alpha )\cap {\rm Min}(A))\cup (D_A(\beta )\cap {\rm Min}(A))=(D_A(\alpha )\cup D_A(\beta ))\cap {\rm Min}(A)=D_A(\alpha \vee \beta )\cap {\rm Min}(A)=D_A(\gamma )\cap {\rm Min}(A)$ iff ${\rm Min}(A)\setminus ((D_A(\alpha )\cap {\rm Min}(A))\cup (D_A(\beta )\cap {\rm Min}(A)))={\rm Min}(A)\setminus (D_A(\gamma )\cap {\rm Min}(A))$ iff $V_A(\alpha )\cap V_A(\beta )\cap {\rm Min}(A)=(V_A(\alpha )\cap {\rm Min}(A))\cap (V_A(\beta )\cap {\rm Min}(A))=V_A(\gamma )\cap {\rm Min}(A)$.

\noindent (\ref{negmin3}) By (\ref{negmin1}) and Remark \ref{eqneg}, $\alpha ^{\perp \perp }=\beta ^{\perp }$ iff $V_A(\alpha ^{\perp \perp })\cap {\rm Min}(A)=V_A(\beta ^{\perp })\cap {\rm Min}(A)$ iff $D_A(\alpha ^{\perp })\cap {\rm Min}(A)=V_A(\alpha )\cap {\rm Min}(A)=V_A(\beta ^{\perp })\cap {\rm Min}(A)=D_A(\beta )\cap {\rm Min}(A)$.

By Lemma \ref{proprperp}, (\ref{proprperp2}), $\alpha ^{\perp \perp }=\beta ^{\perp }$ implies $\alpha ^{\perp }=\alpha ^{\perp \perp \perp }=\beta ^{\perp \perp }$, which also proves the converse.

\noindent {\bf Case 2:} The proof goes similarly in the case when \hypmuA\ holds, but for all $\alpha ,\beta ,\gamma \in {\rm Con}(A)$.\end{proof}

Let us denote by ${\cal F}_{\rm Min}(A)$ the topology on ${\rm Min}(A)$ generated by $\{V_A(a,b)\cap {\rm Min}(A)\ |\ a,b\in A\}$, called the {\em flat topology} or the {\em inverse topology} on ${\rm Min}(A)$. Also, we denote by ${\cal M}in(A)$, respectively ${\cal M}in(A)^{-1}$ the minimal prime spectrum of $A$ endowed with the Stone, respectively the flat topology: ${\cal M}in(A)=({\rm Min}(A),{\cal S}_{\rm Min}(A))$ and ${\cal M}in(A)^{-1}=({\rm Min}(A),{\cal F}_{\rm Min}(A))$.

\begin{remark} ${\cal F}_{\rm Min}(A)$ has $\{V_A(\alpha )\cap {\rm Min}(A)\ |\ \alpha \in {\cal K}(A)\}$ as a basis, since $V_A(\Delta _A)\cap {\rm Min}(A)={\rm Min}(A)$ and, for any $\alpha ,\beta \in {\cal K}(A)$, we have $\alpha \vee \beta \in {\cal K}(A)$ and $V_A(\alpha )\cap {\rm Min}(A)\cap V_A(\beta )\cap {\rm Min}(A)=V_A(\alpha \vee \beta )\cap {\rm Min}(A)$.\end{remark}

Recall that, for any $\alpha \in {\rm Con}(A)$, $\alpha ^{\perp }$ generates the annihilator of $\alpha $ in the commutator lattice $({\rm Con}(A),[\cdot ,\cdot ]_A)$ as a principal ideal.

Let us consider the following condition:\vspace*{3pt}

\noindent \hypcpA\ \ $\alpha ^{\perp }\in {\cal K}(A)$ for any $\alpha \in {\cal K}(A)$.\vspace*{3pt}

Condition \hypcpA\ obviously holds if the lattice ${\rm Con}(A)$ is compact: ${\rm Con}(A)={\cal K}(A)$.

\begin{proposition}\begin{enumerate}
\item\label{coarse1} The flat topology on ${\rm Min}(A)$ is coarser than the Stone topology: ${\cal F}_{\rm Min}(A)\subseteq {\cal S}_{\rm Min}(A)$.
\item\label{coarse2} If \hypcpA\ is satisfied, then the two topologies coincide: ${\cal F}_{\rm Min}(A)={\cal S}_{\rm Min}(A)$, that is ${\cal M}in(A)={\cal M}in(A)^{-1}$.\end{enumerate}\label{coarse}\end{proposition}

\begin{proof} (\ref{coarse1}) By Proposition \ref{negmin}.(\ref{negmin1}), for any $\alpha \in {\cal K}(A)$, $V_A(\alpha )\cap {\rm Min}(A)=D_A(\alpha ^{\perp })\cap {\rm Min}(A)\in {\cal S}_{\rm Min}(A)$.

\noindent (\ref{coarse2}) Again by Proposition \ref{negmin}.(\ref{negmin1}), for any $\alpha \in {\cal K}(A)$, $D_A(\alpha )\cap {\rm Min}(A)=V_A(\alpha ^{\perp })\cap {\rm Min}(A)$, which belongs to ${\cal F}_{\rm Min}(A)$ if $\alpha ^{\perp }\in {\cal K}(A)$.\end{proof}

Now let $L$ be a bounded distributive lattice. Following \cite{retic}, we denote, for any $I\in {\rm Id}(L)$ and $a\in L$, by $V_{{\rm Id},L}(I)={\rm Spec}_{\rm Id}(L)\cap [I)_{{\rm Id}(L)}$, $D_{{\rm Id},L}(I)={\rm Spec}_{\rm Id}(L)\setminus V_{{\rm Id},L}(I)$, $V_{{\rm Id},L}(a)=V_{{\rm Id},L}((a]_L)$ and $D_{{\rm Id},L}(a)=D_{{\rm Id},L}((a]_L)$.

Let us denote by ${\cal S}_{{\rm Spec},{\rm Id}}(L)$ the Stone topology on ${\rm Spec}_{\rm Id}(L)$ and by ${\cal S}_{{\rm Min},{\rm Id}}(L)$ the Stone topology on ${\rm Min}_{\rm Id}(L)$: ${\cal S}_{{\rm Spec},{\rm Id}}(L)=\{D_{{\rm Id},L}(I)\ |\ I\in {\rm Id}(L)\}$, having $\{D_{{\rm Id},L}(a)\ |\ a\in L\}$ as a basis, and ${\cal S}_{{\rm Min},{\rm Id}}(L)=\{D_{{\rm Id},L}(I)\cap {\rm Min}_{\rm Id}(L)\ |\ I\in {\rm Id}(L)\}$, having $\{D_{{\rm Id},L}(a)\cap {\rm Min}_{\rm Id}(L)\ |\ a\in L\}$ as a basis.

Also, let us denote by ${\cal F}_{{\rm Min},{\rm Id}}(L)$ the flat topology on ${\rm Min}_{\rm Id}(L)$, which has $\{V_{{\rm Id},L}(a)\cap {\rm Min}_{\rm Id}(L)\ |\ a\in L\}$ as a basis. And let ${\cal M}in_{\rm Id}(L)$, respectively ${\cal M}in_{\rm Id}(L)^{-1}$ be the minimal prime spectrum of ideals of $L$ endowed with the Stone, respectively the flat topology: ${\cal M}in_{\rm Id}(L)=({\rm Min}_{\rm Id}(L),{\cal S}_{{\rm Min},{\rm Id}}(L))$ and ${\cal M}in_{\rm Id}(L)^{-1}=({\rm Min}_{\rm Id}(L),{\cal F}_{{\rm Min},{\rm Id}}(L))$.

\begin{lemma} If \hyplA\ holds and ${\cal L}(A)$ is the reticulation of $A$, then:\begin{enumerate}
\item\label{homeoMin1} ${\cal M}in(A)$ is homeomorphic to ${\cal M}in_{\rm Id}({\cal L}(A))$;
\item\label{homeoMin2} ${\cal M}in(A)^{-1}$ is homeomorphic to ${\cal M}in_{\rm Id}({\cal L}(A))^{-1}$.\end{enumerate}\label{homeoMin}\end{lemma}

\begin{proof} Assume that $A$ satisfies \hyplA , so that its reticulation can be constructed as: ${\cal L}(A)={\cal K}(A)/\!\equiv _A$. As in \cite{retic}, let us denote by $u:{\rm Spec}(A)\rightarrow {\rm Spec}_{\rm Id}({\cal L}(A))$ and $v:{\rm Spec}_{\rm Id}({\cal L}(A))\rightarrow {\rm Spec}(A)$ the mutually inverse homeomorphisms w.r.t. the Stone topologies mentioned in the proof of Proposition \ref{charmincg}: $u(\phi )=\phi ^{\ast }$ for all $\phi \in {\rm Spec}(A)$ and $v(P)=P_{\ast }$ for all $P\in {\rm Spec}_{\rm Id}({\cal L}(A))$.

\noindent (\ref{homeoMin1}) $u$ and $v$ obviously restrict to homeomorphisms between ${\cal M}in(A)$ and ${\cal M}in_{\rm Id}({\cal L}(A))$.

\noindent (\ref{homeoMin2}) Let us recall the flat topology ${\cal F}_{{\rm Min}(A)}$ has $\{V_A(\alpha )\cap {\rm Min}(A)\ |\ \alpha \in {\cal K}(A)\}$ as a basis, while the flat topology on ${\cal F}_{{\rm Min},{\rm Id}}({\cal L}(A))$ has $\{V_{{\rm Id},{\cal L}(A)}((a]_{{\cal L}(A)})\cap {\rm Min}_{\rm Id}({\cal L}(A))\ |\ a\in {\cal L}(A)\}=\{V_{{\rm Id},{\cal L}(A)}((\alpha /\!\equiv _A]_{{\cal L}(A)})\cap {\rm Min}_{\rm Id}({\cal L}(A))\ |\ \alpha \in {\cal K}(A)\}$ as a basis.

In the proof of \cite[Proposition~$11$]{retic} we have obtained that, for all $\alpha \in {\rm Con}(A)$, $u(V_A(\alpha ))=V_{{\rm Id},{\cal L}(A)}(\alpha ^{\ast })$. Note that, if $\alpha \in {\cal K}(A)$, then $\alpha ^{\ast }=(\alpha /\!\equiv _A]_{{\cal L}(A)}$, thus $u(V_A(\alpha ))=V_{{\rm Id},{\cal L}(A)}((\alpha /\!\equiv _A]_{{\cal L}(A)})$, hence $u$ is open w.r.t. the flat topologies on the minimal prime spectra.

Consequently, for all $\alpha \in {\cal K}(A)$, $v(V_{{\rm Id},{\cal L}(A)}((\alpha /\!\equiv _A]_{{\cal L}(A)}))=v(u(V_A(\alpha )))=V_A(\alpha )$, hence $v$ is open w.r.t. the flat topologies on the minimal prime spectra.

Therefore $u$ and $v$ are mutually inverse homeomorphisms between ${\cal M}in(A)^{-1}$ and ${\cal M}in_{\rm Id}({\cal L}(A))^{-1}$.\end{proof}

\begin{proposition} If \hyplA\ holds, then ${\cal M}in(A)^{-1}$ is a compact ${\rm T}_1$ topological space.\end{proposition}

\begin{proof} Assume that $A$ satisfies \hyplA , and let us consider the reticulation of $A$: ${\cal L}(A)={\cal K}(A)/\!\equiv _A$.

By Hochster's theorem \cite[Proposition~$3.13$]{joh}, there exists a commutative unitary ring $R$ such that the reticulation ${\cal L}(R)$ of $R$ is lattice isomorphic to ${\cal L}(A)$. Recall that the commutator lattice of the ideals of $R$ endowed with the multiplication of ideals as commutator operation is isomorphic to the commutator lattice of its congruences, $({\rm Con}(R),[\cdot ,\cdot ]_R)$.

By Lemma \ref{homeoMin}.(\ref{homeoMin2}), the minimal prime spectrum of $R$ endowed with the flat topology, ${\cal M}in(R)^{-1}$, is homeomorphic to ${\cal M}in_{\rm Id}({\cal L}(R))^{-1}$ and thus to ${\cal M}in_{\rm Id}({\cal L}(A))^{-1}$, which in turn is homeomorphic to ${\cal M}in(A)^{-1}$, thus ${\cal M}in(R)^{-1}$ is homeomorphic to ${\cal M}in(A)^{-1}$.

By \cite[Theorem~$3.1$]{klms}, ${\cal M}in(R)^{-1}$ is compact and ${\rm T}_1$. Therefore ${\cal M}in(A)^{-1}$ is compact and ${\rm T}_1$.\end{proof}

Following \cite{retic}, under \hyplA, we will denote the lattice bounds of ${\cal L}(A)$ by ${\bf 0}$ and ${\bf 1}$, so ${\bf 0}=\nabla _A/\!\equiv _A$ and ${\bf 1}=\Delta _A/\!\equiv _A$.

\begin{theorem} If \hyplA\ is satisfied, then the following are equivalent:\begin{enumerate}
\item\label{charMincp1} ${\cal M}in(A)={\cal M}in(A)^{-1}$;
\item\label{charMincp2} ${\cal M}in(A)$ is compact;
\item\label{charMincp3} for any $\alpha \in {\cal K}(A)$, there exists $\beta \in {\cal K}(A)$ such that $\beta \subseteq \alpha ^{\perp }$ and $(\alpha \vee \beta )^{\perp }=\Delta _A$.\end{enumerate}\label{charMincp}\end{theorem}

\begin{proof} Assume that $A$ satisfies \hyplA. Then the reticulation ${\cal L}(A)$ of $A$ is a bounded distributive lattice and thus a distributive lattice with zero, hence, according to \cite[Proposition~$5.1$]{speed2}, the following are equivalent:

\quad $(a)$ ${\cal M}in_{\rm Id}({\cal L}(A))={\cal M}in_{\rm Id}({\cal L}(A))^{-1}$;

\quad $(b)$ ${\cal M}in_{\rm Id}({\cal L}(A))$ is compact;

\quad $(c)$ for any $x\in {\cal L}(A)$, there exists $y\in {\cal L}(A)$ such that $x\wedge y={\bf 0}$ and ${\rm Ann}_{{\cal L}(A)}(x\vee y)=\{{\bf 0}\}$.

By Lemma \ref{homeoMin}, (\ref{charMincp1}) is equivalent to $(a)$. By Lemma \ref{homeoMin}.(\ref{homeoMin1}), (\ref{charMincp2}) is equivalent to $(b)$.

To prove that (\ref{charMincp3}) is equivalent to $(c)$, let $\alpha ,\beta \in {\cal K}(A)$, arbitrary, so that $\alpha /\!\equiv _A$, and $\beta /\!\equiv _A$ are arbitrary elements of ${\cal L}(A)$.

$A$ is semiprime, that is $\rho _A(\Delta _A)=\Delta _A$, which is equivalent to $\Delta _A/\!\equiv _A=\{\Delta _A\}$ according to \cite[Remark~$5.10$]{stcommlat}, hence, for any $\theta \in {\rm Con}(A)$, $\theta =\Delta _A$ iff $\theta \in \Delta _A/\!\equiv _A$ iff $\theta /\!\equiv _A=\Delta _A/\!\equiv _A$, that is $\theta /\!\equiv _A={\bf 0}$.

Recall that $\beta \subseteq \alpha ^{\perp }$ is equivalent to $[\alpha ,\beta ]_A=\Delta _A$ and thus to $[\alpha ,\beta ]_A/\!\equiv _A={\bf 0}$ by the above, that is $\alpha /\!\equiv _A\wedge \beta /\!\equiv _A={\bf 0}$.

Furthermore, since $A$ is semiprime, we have, for all $\theta \in {\rm Con}(A)$: by \cite[Lemma~$5.18.(ii)$]{stcommlat}, ${\rm Ann}_{({\rm Con}(A),[\cdot ,\cdot ]_A)}(\theta )={\rm Ann}_{{\rm Con}(A)}(\theta )$, and, by Lemma \ref{radeq}.(\ref{radeq1}), ${\rm Ann}_{{\rm Con}(A)}(\theta )=\{\Delta _A\}$ iff ${\rm Ann}_{{\rm Con}(A)/\equiv _A}(\theta /\!\equiv _A)=\{{\bf 0}\}$.

$(\alpha \vee \beta )^{\perp }=\Delta _A$ means that ${\rm Ann}_{({\rm Con}(A),[\cdot ,\cdot ]_A)}(\alpha \vee \beta )=\{\Delta _A\}$, that is ${\rm Ann}_{{\rm Con}(A)}(\alpha \vee \beta )=\{\Delta _A\}$, which is equivalent to ${\rm Ann}_{{\rm Con}(A)/\equiv _A}(\alpha /\!\equiv _A\vee \beta /\!\equiv _A)=\{{\bf 0}\}$, which in turn is equivalent to ${\rm Ann}_{{\cal L}(A)}(\alpha /\!\equiv _A\vee \beta /\!\equiv _A)=\{{\bf 0}\}$, because, if we denote by $\theta =[\alpha ,\beta ]_A$, so that $\theta \in {\cal K}(A)$ and $\theta /\!\equiv _A=\alpha /\!\equiv _A\vee \beta /\!\equiv _A\in {\cal L}(A)$, we have:

since ${\cal L}(A)$ is a bounded sublattice of ${\rm Con}(A)/\!\equiv _A$, ${\rm Ann}_{{\rm Con}(A)/\equiv _A}(\theta /\!\equiv _A)=\{{\bf 0}\}$ implies ${\rm Ann}_{{\cal L}(A)}(\theta /\!\equiv _A)={\rm Ann}_{{\rm Con}(A)/\equiv _A}(\theta /\!\equiv _A)\cap {\cal L}(A)=\{{\bf 0}\}$;

for the converse, recall that $\max {\rm Ann}_{{\rm Con}(A)}(\theta )=\max {\rm Ann}_{({\rm Con}(A),[\cdot ,\cdot ]_A)}(\theta )=\bigvee \{\gamma \in {\cal K}(A)\ |\ [\theta ,\gamma ]_A=\Delta _A\}=\bigvee \{\gamma \in {\cal K}(A)\ |\ [\theta ,\gamma ]_A/\!\equiv _A={\bf 0}\}=\bigvee \{\gamma \in {\cal K}(A)\ |\ \theta /\!\equiv _A\wedge \gamma /\!\equiv _A={\bf 0}\}$, thus, if $\theta \in {\cal K}(A)$, so that $\theta /\!\equiv _A\in {\cal L}(A)$, then $\max {\rm Ann}_{{\rm Con}(A)}(\theta )=\bigvee \{\gamma \in {\cal K}(A)\ |\ \gamma /\!\equiv _A\in {\rm Ann}_{{\cal L}(A)}(\theta /\!\equiv _A)\}$; hence, if ${\rm Ann}_{{\cal L}(A)}(\theta /\!\equiv _A)=\{{\bf 0}\}$, then $\max {\rm Ann}_{{\rm Con}(A)}(\theta )=\bigvee \{\gamma \in {\cal K}(A)\ |\ \gamma /\!\equiv _A\in \{{\bf 0}\}\}=\bigvee \{\gamma \in {\cal K}(A)\ |\ \gamma /\!\equiv _A={\bf 0}\}=\bigvee \{\gamma \in {\cal K}(A)\ |\ \gamma =\Delta _A\}=\Delta _A$, thus ${\rm Ann}_{{\rm Con}(A)}(\theta )=\{\Delta _A\}$, which is equivalent to ${\rm Ann}_{{\rm Con}(A)/\equiv _A}(\theta /\!\equiv _A)=\{{\bf 0}\}$.\end{proof}

\begin{proposition} If $\nabla _A\in {\cal K}(A)$ and ${\rm Spec}(A)$ is unordered, then ${\cal M}in(A)$ is compact.\label{caseMincp}\end{proposition}

\begin{proof} Assume that $\nabla _A\in {\cal K}(A)$ and ${\rm Spec}(A)$ is unordered, that is ${\rm Spec}(A)={\rm Min}(A)$, and let $\displaystyle {\rm Min}(A)=\bigcup _{i\in I}(D_A(\alpha _i)\cap {\rm Min}(A))$ for some nonempty family $\{\alpha _i\ |\ i\in I\}$ of congruences of $A$.

Then $\displaystyle {\rm Min}(A)=(\bigcup _{i\in I}D_A(\alpha _i))\cap {\rm Min}(A)=D_A(\bigvee _{i\in I}\alpha _i)\cap {\rm Min}(A)$, so $V_A(\bigvee _{i\in I}\alpha _i)\cap {\rm Min}(A)=\emptyset $. By Remark \ref{rhomin}, this implies that $\displaystyle \bigvee _{i\in I}\alpha _i=\nabla _A\in {\cal K}(A)$, so that $\displaystyle \nabla _A=\bigvee _{i\in F}\alpha _i$ for some finite subset $F$ of $I$, hence $\displaystyle {\rm Min}(A)=D_A(\bigvee _{i\in F}\alpha _i)\cap {\rm Min}(A)=(\bigcup _{i\in F}D_A(\alpha _i))\cap {\rm Min}(A)=\bigcup _{i\in F}(D_A(\alpha _i)\cap {\rm Min}(A))$, therefore ${\cal M}in(A)$ is compact.\end{proof}

\begin{theorem} If $A$ satisfies \hyplA\ and \hypcpA , in particular if ${\rm Con}(A)={\cal K}(A)$, then ${\cal M}in(A)$ is a Hausdorff topological space consisting solely of clopen sets, thus the Stone topology ${\cal S}_{{\rm Min}(A)}$ is a complete Boolean sublattice of ${\cal P}({\rm Min}(A))$. If, moreover, ${\rm Spec}(A)$ is unordered, then ${\cal M}in(A)$ is also compact.\label{Hausdorff}\end{theorem}

\begin{proof} By Proposition \ref{negmin}, (\ref{negmin1}), the Stone topology ${\cal S}_{\rm Min}(A)$ on ${\rm Min}(A)$ consists entirely of clopen sets.

Let $\mu ,\nu $ be distinct minimal prime congruences of $A$. Then there exist $a,b\in A$ such that $(a,b)\in \mu \setminus \nu $, so that $Cg_A(a,b)\subseteq \mu $ and $Cg_A(a,b)\nsubseteq \nu $, so that $Cg_A(a,b)^{\perp }\nsubseteq \mu $ by Proposition \ref{charmincg}, so $\mu \in D_A(Cg_A(a,b)^{\perp })\cap {\rm Min}(A)$ and $\nu \in D_A(Cg_A(a,b))\cap {\rm Min}(A)$. $D_A(Cg_A(a,b))\cap {\rm Min}(A)\cap D_A(Cg_A(a,b)^{\perp })\cap {\rm Min}(A)=D_A(Cg_A(a,b))\cap D_A(Cg_A(a,b)^{\perp })\cap {\rm Min}(A)=D_A([Cg_A(a,b),Cg_A(a,b)^{\perp }]_A)\cap {\rm Min}(A)=D_A(\Delta _A)\cap {\rm Min}(A)=\emptyset \cap {\rm Min}(A)=\emptyset $, therefore the topological space $({\rm Min}(A),\{D_A(\theta )\cap {\rm Min}(A)\ |\ \theta \in {\rm Con}(A)\})$ is Hausdorff.

By Proposition \ref{caseMincp}, if ${\rm Spec}(A)$ is an antichain, then ${\cal M}in(A)$ is also compact.\end{proof}

\section{$m$--extensions}
\label{algext}

\hnosp Throughout this section, we will assume that $A$ is a subalgebra of $B$, the algebras $A$ and $B$ are semiprime and the commutators of $A$ and $B$ are commutative and distributive w.r.t. arbitrary joins.

In particular, the following results hold for extensions of semiprime algebras in congruence--modular varieties.

To avoid any danger of confusion, we will denote by $\alpha ^{\perp A}=\alpha \rightarrow \Delta _A$ and $X^{\perp A}=Cg_A(X)^{\perp A}$ for any $\alpha \in {\rm Con}(A)$ and any $X\subseteq A^2$ and by $\beta ^{\perp B}=\beta \rightarrow \Delta _B$ and $Y^{\perp B}=Cg_B(Y)^{\perp B}$ for any $\beta \in {\rm Con}(B)$ and any $Y\subseteq B^2$. See this notation for arbitrary subsets in Section \ref{rescglat}.\vspace*{7pt}

We call the extension $A\subseteq B$:\begin{itemize}
\item {\em admissible} iff the map $i_{A,B}:A\rightarrow B$ is admissible, that is iff $i_{A,B}^*(\phi )=\phi \cap \nabla _A\in {\rm Spec}(A)$ for all $\phi \in {\rm Spec}(B)$;
\item {\em ${\rm Min}$--admissible} or an {\em $m$--extension} iff $i_{A,B}^*(\mu )=\mu \cap \nabla _A\in {\rm Min}(A)$ for all $\phi \in {\rm Min}(B)$.\end{itemize}

\begin{lemma} Assume that the extension $A\subseteq B$ is admissible and let us consider the following statements:\begin{enumerate}
\item\label{mext1} $A\subseteq B$ is an $m$--extension;
\item\label{mext2} for any $\alpha \in {\cal K}(A)$ and any $\mu \in {\rm Min}(B)$, if $\alpha \subseteq \mu $, then $\alpha ^{\perp A}\nsubseteq \mu $;
\item\label{mext3} for any $\alpha \in {\cal K}(A)$ and any $\mu \in {\rm Min}(B)$, $\alpha \subseteq \mu $ iff $\alpha ^{\perp A}\nsubseteq \mu $;
\item\label{mext4} for any $\alpha \in {\rm Con}(A)$ and any $\mu \in {\rm Min}(B)$, if $\alpha \subseteq \mu $, then $\alpha ^{\perp A}\nsubseteq \mu $;
\item\label{mext5} for any $\alpha \in {\rm Con}(A)$ and any $\mu \in {\rm Min}(B)$, $\alpha \subseteq \mu $ iff $\alpha ^{\perp A}\nsubseteq \mu $.\end{enumerate}

If $A$ satisfies \hyplA , then (\ref{mext1}), (\ref{mext2}) and (\ref{mext3}) are equivalent.

If $A$ satisfies \hypmuA , then (\ref{mext1}), (\ref{mext4}) and (\ref{mext5}) are equivalent.\label{mext}\end{lemma}

\begin{proof} For any $\alpha \in {\rm Con}(A)$ and $\mu \in {\rm Con}(B)$, we obviously have: $\alpha \subseteq \mu $ iff $\alpha \subseteq \mu \cap \nabla _A$, and $\alpha ^{\perp A}\nsubseteq \mu $ iff $\alpha ^{\perp A}\nsubseteq \mu \cap \nabla _A$.

Now assume that $A$ satisfies \hyplA\ or \hypmuA , and let $M={\cal K}(A)$ if \hyplA\ holds, and $M={\rm Con}(A)$ if \hypmuA\ holds.

Since the extension $A\subseteq B$ is admissible, we have $\mu \cap \nabla _A\in {\rm Spec}(A)$ for any $\mu \in {\rm Spec}(B)$. $A\subseteq B$ is an $m$--extension iff $\mu \cap \nabla _A\in {\rm Min}(A)$ for any $\mu \in {\rm Min}(B)$, hence, by Proposition \ref{charmincg}: $A\subseteq B$ is an $m$--extension iff, for all $\mu \in {\rm Min}(B)$ and all $\alpha in M$, the following equivalence holds: $\alpha \subseteq \mu \cap \nabla _A $ iff $\alpha ^{\perp A}\nsubseteq \mu \cap \nabla _A$; by the above, this is equivalent to: $\alpha \subseteq \mu $ iff $\alpha ^{\perp A}\nsubseteq \mu $.\end{proof}

If $A\subseteq B$ is an $m$--extension, then the function $\Gamma =i_{A,B}^*\mid _{{\rm Min}(B)}:{\rm Min}(B)\rightarrow {\rm Min}(A)$, $\Gamma (\mu )=\mu \cap \nabla _A$ for all $\mu \in {\rm Min}(B)$, is well defined.

\begin{proposition} If the extension $A\subseteq B$ is admissible, then, for every $\psi \in {\rm Spec}(A)$, there exists a $\mu \in {\rm Min}(B)$ such that $\mu \cap \nabla _A\subseteq \psi $.\end{proposition}

\begin{proof} Since $\psi \in {\rm Spec}(A)$, $\nabla _A\setminus \psi $ is an $m$--system in $A$, thus also in $B$, according to \cite[Lemma $4.18$]{gulo}. Hence there exists a $\nu \in {\rm Max}\{\gamma \in {\rm Con}(B)\ |\ \gamma \cap 
(\nabla _A\setminus \psi )=\emptyset \}$, so that $\nu \in {\rm Spec}(B)$ by Lemma \ref{msistspec}, and thus there exists a $\mu \in {\rm Min}(B)$ with $\mu \subseteq \nu $, so that $\mu \cap (\nabla _A\setminus \psi )\subseteq \nu \cap (\nabla _A\setminus \psi )=\emptyset $ and thus $(\mu \cap \nabla _A)\setminus \psi =\emptyset $, so $\mu \cap \nabla _A\subseteq \psi $.\end{proof}

\begin{corollary}\begin{itemize}
\item If the extension $A\subseteq B$ is admissible, then, for every $\psi \in {\rm Min}(A)$, there exists a $\mu \in {\rm Min}(B)$ such that $\mu \cap \nabla _A=\psi $.
\item If $A\subseteq B$ is an admissible $m$--extension, then $\Gamma :{\rm Min}(B)\rightarrow {\rm Min}(A)$ is surjective.
\end{itemize}\label{Gammasurj}\end{corollary}

\begin{lemma} If $A\subseteq B$ is an admissible $m$--extension, then, for any $\theta ,\zeta \in {\rm Con}(A)$: $[\theta ,\zeta ]_A=\Delta _A$ iff $[Cg_B(\theta ),Cg_B(\zeta )]_B=\Delta _B$.\label{eqcomm}\end{lemma}

\begin{proof} Since $A$ and $B$ are semiprime, that is $\Delta _A=\rho _A(\Delta _A)=\bigcap {\rm Min}(A)$ and $\Delta _B=\rho _B(\Delta _B)=\bigcap {\rm Min}(B)$, we have:

$[Cg_B(\theta ),Cg_B(\zeta )]_B=\Delta _B$  iff $[Cg_B(\theta ),Cg_B(\zeta )]_B\subseteq \nu $ for all $\nu \in {\rm Min}(B)$ iff, for all $\nu \in {\rm Min}(B)$, $Cg_B(\theta )\subseteq \nu $ or $Cg_B(\zeta )\subseteq \nu $ iff, for all $\nu \in {\rm Min}(B)$, $\theta \subseteq \nu $ or $\zeta \subseteq \nu $ iff, for all $\nu \in {\rm Min}(B)$, $\theta \subseteq \nu \cap \nabla _A$ or $\zeta \subseteq \nu \cap \nabla _A$; by Corollary \ref{Gammasurj}, the latter is equivalent to: for all $\mu \in {\rm Min}(A)$, $\theta \subseteq \mu $ or $\zeta \subseteq \mu $, which in turn is equivalent to the fact that $[\theta ,\zeta ]_A\subseteq \mu $ for all $\mu \in {\rm Min}(A)$, that is $[\theta ,\zeta ]_A=\Delta _A$.\end{proof}

\begin{proposition} If $A\subseteq B$ is an admissible $m$--extension, then, for any $\theta \in {\rm Con}(A)$: $\theta ^{\perp A}=\theta ^{\perp B}\cap \nabla _A$ and $\theta ^{\perp B}=Cg_B(\theta ^{\perp A})$.\label{intersperp}\end{proposition}

\begin{proof} By Lemma \ref{eqcomm} we have, for any $u,v\in A$: $(u,v)\in \theta ^{\perp A}$ iff $[Cg_A(u,v),\theta ]_A=\Delta _A$ iff $[Cg_B(Cg_A(u,v)),Cg_B(\theta )]_B=[Cg_B(u,v),Cg_B(\theta )]_B=\Delta _B$ iff $(u,v)\in \theta ^{\perp B}$ iff $(u,v)\in \theta ^{\perp B}\cap \nabla _A$.

Again by Lemma \ref{eqcomm}, we have, for any $\beta \in {\rm Con}(B)$: $\theta ^{\perp B}=Cg_B(\theta )^{\perp B}=\max \{\beta \in {\rm Con}(B)\ |\ [\beta ,Cg_B(\theta )]_B=\Delta _B\}\subseteq \max \{Cg_B(\alpha )\ |\ \alpha \in {\rm Con}(A),[Cg_B(\alpha ),Cg_B(\theta )]_B=\Delta _B\}=\max \{Cg_B(\alpha )\ |\ \alpha \in {\rm Con}(A),\alpha ,\theta ]_A=\Delta _A\}=Cg_B(\max \{\alpha \in {\rm Con}(A)\ |\ \alpha ,\theta ]_A=\Delta _A\}=Cg_B(\theta ^{\perp A})$; on the other hand: $[\theta ,\theta ^{\perp A}]_A=\Delta _A$, thus $[Cg_B(\theta ),Cg_B(\theta ^{\perp A})]_B=\Delta _B$, so $Cg_B(\theta ^{\perp A})\subseteq Cg_B(\theta )^{\perp B}=\theta ^{\perp B}$. Therefore $Cg_B(\theta ^{\perp A})=\theta ^{\perp B}$.\end{proof}

\begin{corollary} If $A\subseteq B$ is an admissible $m$--extension, then, for any $\theta ,\zeta \in {\rm Con}(A)$: $\theta ^{\perp A}=\zeta ^{\perp A}$ iff $\theta ^{\perp B}=\zeta ^{\perp B}$.\label{eqperp}\end{corollary}

\begin{corollary} If $A\subseteq B$ is an admissible $m$--extension such that $A$ satisfies \hypmuA\ and $B$ satisfies \hyplB\ or \hypmuB , then, for any $\psi \in {\rm Spec}(B)$, we have: $\psi \in {\rm Min}(B)$ iff $\psi \cap \nabla _A\in {\rm Min}(A)$.\label{admMiniff}\end{corollary}

\begin{proof} We have the direct implication by the definition of an $m$--extension.

Now assume that $\psi \cap \nabla _A\in {\rm Min}(A)$ and let $\beta \in {\rm Con}(B)$, arbitrary. Then, by Propositions \ref{charmincg} and \ref{intersperp}: $\beta \subseteq \psi $ implies $\beta \cap\nabla _A\subseteq \psi \cap \nabla _A$, which is equivalent to $(\beta \cap\nabla _A)^{\perp A}\nsubseteq \psi \cap \nabla _A$, hence $(\beta \cap\nabla _A)^{\perp A}\nsubseteq \psi $, thus $Cg_B(\beta \cap\nabla _A)^{\perp B}=(\beta \cap\nabla _A)^{\perp B}=Cg_B((\beta \cap\nabla _A)^{\perp A})\nsubseteq \psi $, hence $[\beta ,\psi ]_B\supseteq [Cg_B(\beta \cap\nabla _A),\psi ]_B\supsetneq \Delta _B$ since $\beta \supseteq Cg_B(\beta \cap\nabla _A)$, thus $\beta ^{\perp B}\nsubseteq \psi $. Therefore, again by Proposition \ref{charmincg}, $\psi \in {\rm Min}(B)$.\end{proof}

\begin{remark} Note from the proof of Corollary \ref{admMiniff} that, if $A\subseteq B$ is an admissible $m$--extension such that $A$ satisfies \hypmuA\ and $B$ satisfies \hyplB\ or \hypmuB , then $B$ satisfies the equivalence of all statements (\ref{1phiMin}), (\ref{2K}), (\ref{3K}), (\ref{2Con}) and (\ref{3Con}) in Proposition \ref{charmincg}.\end{remark}

By extending the terminology for ring extensions from \cite{bhadremcgov}, we call $A\subseteq B$:\begin{itemize}
\item a {\em rigid}, respectively {\em quasirigid}, respectively {\em weak rigid extension} iff, for any $\beta \in {\rm PCon}(B)$, there exists an $\alpha \in {\rm PCon}(A)$, respectively an $\alpha \in {\cal K}(A)$, respectively an $\alpha \in {\rm Con}(A)$ such that $\alpha ^{\perp B}=\beta ^{\perp B}$;
\item an {\em $r$--extension}, respectively a {\em quasi $r$--extension}, respectively {\em weak $r$--extension} iff, for any $\mu \in {\rm Min}(B)$ and any $\beta \in {\rm PCon}(B)$ such that $\beta \nsubseteq \mu $, there exists an $\alpha \in {\rm PCon}(A)$, respectively an $\alpha \in {\cal K}(A)$, respectively an $\alpha \in {\rm Con}(A)$ such that $\alpha \nsubseteq \mu $ and $\beta ^{\perp B}\subseteq \alpha ^{\perp B}$;
\item an {\em $r^*$--extension}, respectively a {\em quasi $r^*$--extension}, respectively a {\em weak $r^*$--extension} iff, for any $\mu \in {\rm Min}(B)$ and any $\beta \in {\rm PCon}(B)$ such that $\beta \subseteq \mu $, there exists an $\alpha \in {\rm PCon}(A)$, respectively an $\alpha \in {\cal K}(A)$, respectively an $\alpha \in {\rm Con}(A)$ such that $\alpha \subseteq \mu $ and $\alpha ^{\perp B}\subseteq \beta ^{\perp B}$.\end{itemize}

\begin{remark} If $A\subseteq B$ is admissible or an $m$--extension, then, since any $\alpha \in {\rm Con}(A)$ and $\mu \in {\rm Con}(B)$ satisfy the equivalence $\alpha \subseteq \mu $ iff $\alpha \subseteq \mu \cap \nabla _A$, thus also the equivalence $\alpha \nsubseteq \mu $ iff $\alpha \nsubseteq \mu \cap \nabla _A$, it follows that $A\subseteq B$ is:\begin{itemize}
\item an $r$--extension, respectively a quasi $r$--extension, respectively weak $r$--extension iff, for any $\beta \in {\rm PCon}(B)$, $\{\mu \cap \nabla _A\ |\ \mu \in D_B(\beta )\cap {\rm Min}(B)\}\subseteq \bigcup \{D_A(\alpha )\ |\ \alpha \in M,\beta ^{\perp B}\subseteq \alpha ^{\perp B}\}=D_A(\bigvee \{\alpha \in M\ |\ \beta ^{\perp B}\subseteq \alpha ^{\perp B}\})$, where $M$ is equal to ${\rm PCon}(A)$, respectively ${\cal K}(A)$, respectively ${\rm Con}(A)$;
\item an $r^*$--extension, respectively a quasi $r^*$--extension, respectively a weak $r^*$--extension iff, for any $\beta \in {\rm PCon}(B)$, $\{\mu \cap \nabla _A\ |\ \mu \in V_B(\beta )\cap {\rm Min}(B)\}\subseteq \bigcup \{V_A(\alpha )\ |\ \alpha \in M,\alpha ^{\perp B}\subseteq \beta ^{\perp B}\}$, where $M$ is equal to ${\rm PCon}(A)$, respectively ${\cal K}(A)$, respectively ${\rm Con}(A)$;\end{itemize}
thus, if $A\subseteq B$ is an $m$--extension, then $A\subseteq B$ is:\begin{itemize}
\item an $r$--extension, respectively a quasi $r$--extension, respectively weak $r$--extension iff, for any $\beta \in {\rm PCon}(B)$, $\Gamma (D_B(\beta )\cap {\rm Min}(B))\subseteq \bigcup \{D_A(\alpha )\ |\ \alpha \in M,\beta ^{\perp B}\subseteq \alpha ^{\perp B}\}=D_A(\bigvee \{\alpha \in M\ |\ \beta ^{\perp B}\subseteq \alpha ^{\perp B}\})$, where $M$ is equal to ${\rm PCon}(A)$, respectively ${\cal K}(A)$, respectively ${\rm Con}(A)$;
\item an $r^*$--extension, respectively a quasi $r^*$--extension, respectively a weak $r^*$--extension iff, for any $\beta \in {\rm PCon}(B)$, $\Gamma (V_B(\beta )\cap {\rm Min}(B))\subseteq \bigcup \{V_A(\alpha )\ |\ \alpha \in M,\alpha ^{\perp B}\subseteq \beta ^{\perp B}\}$, where $M$ is equal to ${\rm PCon}(A)$, respectively ${\cal K}(A)$, respectively ${\rm Con}(A)$.\end{itemize}\label{charext}\end{remark}

\begin{remark} Note from Lemma \ref{fcg} that, for any set $I$ and any $\{a_i,b_i\ |\ i\in I\}\subseteq A$, $Cg_B(Cg_A(\{(a_i,b_i)\ |\ i\in I\}))=Cg_B(\{(a_i,b_i)\ |\ i\in I\})$, hence, for any $\alpha \in {\rm PCon}(A)$, $\beta \in {\cal K}(A)$ and $\gamma \in {\rm Con}(A)$, it follows that $Cg_B(\alpha )\in {\rm PCon}(B)$, $Cg_B(\beta )\in {\cal K}(B)$ and $Cg_B(\gamma )\in {\rm Con}(B)$.\label{cggenext}\end{remark}

\begin{proposition} If $A\subseteq B$ is an $m$--extension, then:\begin{enumerate}
\item\label{rigid2} if $B$ satisfies \hypmuB\ and $A\subseteq B$ is a weak rigid extension, then it is both a weak $r$--extension and a weak $r^*$--extension;
\item\label{rigid0} if $B$ satisfies \hyplB\ or \hypmuB\ and $A\subseteq B$ is a quasirigid extension, then it is both a quasi $r$--extension and a quasi $r^*$--extension;
\item\label{rigid1} if $B$ satisfies \hyplB\ or \hypmuB\ and $A\subseteq B$ is a rigid extension, then it is both an $r$--extension and an $r^*$--extension.
\end{enumerate}\label{rigid}\end{proposition}

\begin{proof} (\ref{rigid0}) Assume that $A\subseteq B$ is a quasirigid extension and let $\mu \in {\rm Min}(B)$ and $\beta \in {\rm PCon}(B)$, so that $Cg_B(\alpha )^{\perp B}=\beta ^{\perp B}$ for some $\alpha \in {\cal K}(A)$, hence, according to Proposition \ref{charmincg} and Remark \ref{cggenext}:\begin{itemize}
\item $\beta \nsubseteq \mu $ implies $Cg_B(\alpha )^{\perp B}=\beta ^{\perp B}\subseteq \mu $, thus $Cg_B(\alpha )\nsubseteq \mu $, hence $\alpha \nsubseteq \mu $;
\item $\beta \subseteq \mu $ implies $Cg_B(\alpha )^{\perp B}=\beta ^{\perp B}\nsubseteq \mu $, thus $\alpha \subseteq Cg_B(\alpha )\subseteq \mu $.\end{itemize}

\noindent (\ref{rigid2}) and (\ref{rigid1}) Analogously.\end{proof}

\begin{proposition} If $A\subseteq B$ is an $m$--extension, then:\begin{enumerate}
\item\label{Gammacont0} $\Gamma $ is continuous w.r.t. the Stone topologies and the inverse topologies;
\item\label{Gammacont1} if \hyplA ,\hypmuA ,\hyplB\ or \hypmuB\ is satisfied, then $\Gamma :{\cal M}in(B)\rightarrow {\cal M}in(A)^{-1}$ is continuous.
\item\label{Gammacont2} if $A$ satisfies \hypcpA , along with one of \hyplA\ or \hypmuA , or $B$ satisfies \hypcpB , along with one of \hyplB\ or \hypmuB , then $\Gamma :{\cal M}in(B)^{-1}\rightarrow {\cal M}in(A)$ is continuous.\end{enumerate}\label{Gammacont}\end{proposition}

\begin{proof} Let $\alpha \in {\rm Con}(A)$, so that $D_A(\alpha )\cap {\rm Min}(A)$ is an arbitrary open set in ${\cal M}in(A)$, $D_B(Cg_B(\alpha ))\cap {\rm Min}(B)$ is an open set in ${\cal M}in(B)$ and, if $\alpha \in {\cal K}(A)$, so that $Cg_B(\alpha )\in {\cal K}(B)$, then $V_A(\alpha )\cap {\rm Min}(A)$ is an arbitrary basic open set in ${\cal M}in(A)^{-1}$ and $V_B(Cg_B(\alpha ))\cap {\rm Min}(B)$ is a basic open set in ${\cal M}in(B)^{-1}$.

Since $A\subseteq B$ is an $m$--extension, we have, for all $\nu \in {\rm Min}(B)$: $\nu \cap \nabla _A\in {\rm Min}(A)$, thus:

$\nu \in \Gamma ^{-1}(V_A(\alpha )\cap {\rm Min}(A))$ iff $\nu \cap \nabla _A\in V_A(\alpha )\cap {\rm Min}(A))=[\alpha )\cap {\rm Min}(A)$ iff $\nu \cap \nabla _A\in [\alpha )$ iff $\alpha \subseteq \nu \cap \nabla _A$ iff $\alpha \subseteq \nu $ iff $Cg_B(\alpha )\subseteq \nu $ iff $\nu \in V_B(Cg_B(\alpha ))\cap {\rm Min}(B)$; hence $\Gamma ^{-1}(V_A(\alpha )\cap {\rm Min}(A))=V_B(Cg_B(\alpha ))\cap {\rm Min}(B)$;

similarly, $\nu \in \Gamma ^{-1}(D_A(\alpha )\cap {\rm Min}(A))$ iff $\alpha \nsubseteq \nu \cap \nabla _A$ iff $\alpha \nsubseteq \nu $ iff $Cg_B(\alpha )\nsubseteq \nu $ iff $\nu \in D_B(Cg_B(\alpha ))\cap {\rm Min}(B)$; hence $\Gamma ^{-1}(D_A(\alpha )\cap {\rm Min}(A))=D_B(Cg_B(\alpha ))\cap {\rm Min}(B)$.

\noindent (\ref{Gammacont0}) Hence $\Gamma :{\cal M}in(B)\rightarrow {\cal M}in(A)$ and $\Gamma :{\cal M}in(B)^{-1}\rightarrow {\cal M}in(A)^{-1}$  are continuous.

\noindent (\ref{Gammacont1}) Assume that $\alpha \in {\cal K}(A)$.

If \hyplA\ or \hypmuA hold, then, by Proposition \ref{charmincg}: $\alpha \subseteq \nu \cap \nabla _A$ iff $\alpha ^{\perp A}\nsubseteq \nu \cap \nabla _A$ iff $\alpha ^{\perp A}\nsubseteq \nu $ iff $Cg_B(\alpha ^{\perp A})\nsubseteq \nu $ iff $\nu \in D_B(Cg_B(\alpha ^{\perp A}))\cap {\rm Min}(B)$; hence $\Gamma ^{-1}(V_A(\alpha )\cap {\rm Min}(A))=D_B(Cg_B(\alpha ^{\perp A}))\cap {\rm Min}(B)$.

If \hyplB\ or \hypmuB hold, then, by Proposition \ref{charmincg}: $Cg_B(\alpha )\subseteq \nu $ iff $Cg_B(\alpha )^{\perp B}\nsubseteq \nu $ iff  $\nu \in D_B(Cg_B(\alpha )^{\perp B})\cap {\rm Min}(B)$; hence $\Gamma ^{-1}(V_A(\alpha )\cap {\rm Min}(A))=D_B(Cg_B(\alpha )^{\perp B})\cap {\rm Min}(B)$.

Hence, in either of these cases, $\Gamma :{\cal M}in(B)\rightarrow {\cal M}in(A)^{-1}$ is continuous.

\noindent (\ref{Gammacont2}) Analogous to the proof of (\ref{Gammacont1}) or simply by applying (\ref{Gammacont0}), (\ref{Gammacont1}) and Proposition \ref{coarse}.(\ref{coarse2}).\end{proof}


\begin{proposition} If $A\subseteq B$ is an admissible quasi $r$--extension and $B$ satisfies \hyplB\ or \hypmuB , then:\begin{itemize}
\item $A\subseteq B$ is an $m$--extension;
\item if, furthermore, $A$ satisfies \hyplA\ or \hypmuA , then $\Gamma $ is a bijection.\end{itemize}\label{rGammabij}\end{proposition}

\begin{proof} Assume that $A\subseteq B$ is an admissible quasi $r$--extension and $B$ satisfies \hyplB\ or \hypmuB .

Assume by absurdum that there exists a $\nu \in {\rm Min}(B)$ with $\nu \cap \nabla _A\notin {\rm Min}(A)$, so that $\nu \cap \nabla _A\in {\rm Spec}(A)\setminus {\rm Min}(A)$ since ${\rm Min}(B)\subseteq {\rm Spec}(B)$ and $A\subseteq B$ is admissible, hence there exists $\mu \in {\rm Min}(A)$ such that $\mu \subsetneq \nu \cap \nabla _A$.

Since $A\subseteq B$ is admissible, by Corollary \ref{Gammasurj} it follows that $\mu =\varepsilon \cap \nabla _A$ for some $\varepsilon \in {\rm Min}(B)$. Thus $\varepsilon \cap \nabla _A=\mu \subsetneq \nu \cap \nabla _A$, therefore $\varepsilon $ and $\nu $ are distinct minimal prime congruences of $B$, hence they are incomparable, thus $\varepsilon \setminus \nu \neq \emptyset $, so that $(x,y)\in \varepsilon \setminus \nu $ for some $x,y\in B$. 

Then $Cg_B(x,y)\nsubseteq \nu $, so that, since $A\subseteq B$ is a quasi $r$--extension, there exists an $\alpha \in {\cal K}(A)$ such that $\alpha \nsubseteq \nu $ and $Cg_B(x,y)^{\perp B}\subseteq \alpha ^{\perp B}=Cg_B(\alpha )^{\perp B}$. Then $Cg_B(\alpha )\in {\cal K}(B)$ and $\alpha \nsubseteq \nu $, thus $Cg_B(\alpha )\nsubseteq \mu $, hence $Cg_B(\alpha )^{\perp B}\subseteq \nu $ by Proposition \ref{charmincg} and the fact that $B$ satisfies \hyplB\ or \hypmuB .

Also, $Cg_B(x,y)\subseteq \varepsilon $, thus, again by Proposition \ref{charmincg} and the fact that $B$ satisfies \hyplB\ or \hypmuB , $Cg_B(\alpha )^{\perp B}\supseteq Cg_B(x,y)^{\perp B}\nsubseteq \varepsilon $, hence $\alpha \subseteq Cg_B(\alpha )\subseteq \varepsilon $, thus $\alpha \subseteq \varepsilon \cap \nabla _A=\mu \subset \nu \cap \nabla _A\subseteq \nu $, hence $Cg_B(\alpha )\subseteq \nu $, so that $Cg_B(\alpha )^{\perp B}\nsubseteq \nu $, which contradicts the above.

Therefore $A\subseteq B$ is an $m$--extension, hence $\Gamma $ is surjective by Corollary \ref{Gammasurj} and the admissibility of $A\subseteq B$.

Now let $\phi ,\psi \in {\rm Min}(B)$ such that $\Gamma (\phi )=\Gamma (\psi )$, that is $\phi \cap \nabla _A=\psi \cap \nabla _A$, and assume by absurdum that $\phi \neq \psi $, so that $\phi \setminus \psi \neq \emptyset $, that is $(u,v)\in \phi \setminus \psi $ for some $u,v\in B$, which thus satisfy $Cg_B(u,v)\subseteq \phi $ and $Cg_B(u,v)\nsubseteq \psi $.

As above, it follows that there exists a $\gamma \in {\cal K}(A)$ such that $\gamma \nsubseteq \psi $ and $Cg_B(u,v)^{\perp B}\subseteq \gamma ^{\perp B}=Cg_B(\gamma ^{\perp A})$ by Proposition \ref{intersperp}, while $Cg_B(u,v)^{\perp B}\nsubseteq \phi $, hence $Cg_B(\gamma ^{\perp A})\nsubseteq \phi $, thus $\gamma ^{\perp A}\nsubseteq \phi $. But then $\gamma \nsubseteq \psi \cap \nabla _A\in {\rm Min}(A)$ and $\gamma ^{\perp A}\nsubseteq \phi \cap \nabla _A\in {\rm Min}(A)$. If $A$ satisfies \hyplA\ or \hypmuA , then this implies $\gamma \subseteq \phi \cap \nabla _A$ by Proposition \ref{charmincg}.

Hence $\gamma \subseteq \phi \cap \nabla _A=\psi \cap \nabla _A\nsupseteq \gamma $; a contradiction. Therefore $\Gamma $ is injective.\end{proof}

\begin{proposition} If $A\subseteq B$ is an admissible quasi $r^{\ast}$--extension and $B$ satisfies \hyplB\ or \hypmuB , then:\begin{itemize}
\item $A\subseteq B$ is an $m$--extension;
\item if, furthermore, $A$ satisfies \hyplA\ or \hypmuA , then $\Gamma $ is a bijection.\end{itemize}\label{r*Gammabij}\end{proposition}

\begin{proof} Similar to the proof of Proposition \ref{rGammabij}.\end{proof}

\begin{theorem} If $A\subseteq B$ is an admissible extension such that $A$ satisfies \hyplA\ or \hypmuA\ and $B$ satisfies \hyplB\ or \hypmuB, then the following are equivalent:\begin{enumerate}
\item\label{GammahomStone1} $A\subseteq B$ is an $r$--extension;
\item\label{GammahomStone2} $A\subseteq B$ is a quasi $r$--extension;
\item\label{GammahomStone3} $\Gamma:{\cal M}in(B)\rightarrow {\cal M}in(A)$ is a homeomorphism.\end{enumerate}\label{GammahomStone}\end{theorem}

\begin{proof} (\ref{GammahomStone1})$\Rightarrow $(\ref{GammahomStone2}): Trivial.

\noindent (\ref{GammahomStone2})$\Rightarrow $(\ref{GammahomStone3}): If $A\subseteq B$ is an admissible quasi $r$--extension such that $A$ satisfies \hyplA\ or \hypmuA\ and $B$ satisfies \hyplB\ or \hypmuB, then, by Propositions \ref{rGammabij} and \ref{Gammacont}, it follows that $A\subseteq B$ is an $m$--extension and $\Gamma :{\cal M}in(B)\rightarrow {\cal M}in(A)$ is a continuous bijection.

Let $\beta \in {\rm PCon}(B)$, arbitrary, so that $\Gamma (D_B(\beta )\cap {\rm Min}(B))=\{\psi \cap \nabla _A\ |\  \psi \in D_B(\beta )\cap {\rm Min}(B)\}=\{\psi \cap \nabla _A\ |\  \psi \in {\rm Min}(B), \beta \nsubseteq \psi \}$.

By Remark \ref{charext} and the fact that $A\subseteq B$ is a quasi $r$--extension and an $m$--extension, $\Gamma (D_B(\beta )\cap {\rm Min}(B))\subseteq D_A(\bigvee \{\alpha \in {\cal K}(A)\ |\ \beta ^{\perp B}\subseteq \alpha ^{\perp B}\})\cap {\rm Min}(A)$.

Now let $\phi \in D_A(\bigvee \{\alpha \in {\cal K}(A)\ |\ \beta ^{\perp B}\subseteq \alpha ^{\perp B}\})\cap {\rm Min}(A)=(\bigcup \{D_A(\alpha )\ |\ \alpha \in {\cal K}(A),\beta ^{\perp B}\subseteq \alpha ^{\perp B}\})\cap {\rm Min}(A)=\bigcup \{D_A(\alpha )\cap {\rm Min}(A)\ |\ \alpha \in {\cal K}(A),\beta ^{\perp B}\subseteq \alpha ^{\perp B}\}$, so that $\phi \in D_A(\alpha )\cap {\rm Min}(A)$ for some $\alpha \in {\cal K}(A)$ such that $\beta ^{\perp B}\subseteq \alpha ^{\perp B}=(Cg_B(\alpha))^{\perp B}$ by Proposition \ref{intersperp}.

By Corollary \ref{Gammasurj}, there exists $\psi \in {\rm Min}(B)$ such that $\Gamma (\psi )=\psi \cap \nabla _A=\phi \nsupseteq \alpha $, thus $\psi \nsupseteq Cg_B(\alpha)$, hence $\psi \supseteq Cg_B(\alpha)^{\perp B}=\beta ^{\perp B}$ and thus $\psi \nsupseteq \beta $ by Proposition \ref{charmincg} and the fact that $B$ satisfies \hyplB\ or \hypmuB , so that $\psi \in D_B(\beta )\cap {\rm Min}(B)$, thus $\phi \in \Gamma (D_B(\beta )\cap {\rm Min}(B))$.

Hence we also have the converse inclusion: $D_A(\bigvee \{\alpha \in {\cal K}(A)\ |\ \beta ^{\perp B}\subseteq \alpha ^{\perp B}\})\cap {\rm Min}(A)\subseteq \Gamma (D_B(\beta )\cap {\rm Min}(B))$, hence $\Gamma (D_B(\beta )\cap {\rm Min}(B))=D_A(\bigvee \{\alpha \in {\cal K}(A)\ |\ \beta ^{\perp B}\subseteq \alpha ^{\perp B}\})\cap {\rm Min}(A)$.

Therefore $\Gamma :{\cal M}in(B)\rightarrow {\cal M}in(A)$ is also open, thus it is a homeomorphism.

\noindent (\ref{GammahomStone3})$\Rightarrow $(\ref{GammahomStone1}): Assume that $\Gamma $ is a homeomorphism w.r.t. the Stone topologies, so that $A\subseteq B$ is an $m$--extension and $\Gamma $ maps basic open sets of ${\cal M}in(B)$ to basic open sets of ${\cal M}in(A)$.

Let $\beta \in {\rm PCon}(B)$, so that $D_B(\beta )\cap {\rm Min}(B)$ is a basic open set of ${\cal M}in(B)$. By the above, there exists $\alpha \in {\rm PCon}(A)$ such that $\{\mu \cap \nabla _A\ |\ \mu \in D_B(\beta )\cap {\rm Min}(B)\}=\Gamma (D_B(\beta )\cap {\rm Min}(B))=D_A(\alpha )\cap {\rm Min}(A)$.

Hence, for all $\mu \in D_B(\beta )\cap {\rm Min}(B)$, that is $\mu \in {\rm Min}(B)$ such that $\beta \nsubseteq \mu $, we have $\mu \cap \nabla _A\in D_A(\alpha )\cap {\rm Min}(A)$, so $\alpha \nsubseteq \mu \cap \nabla _A$, that is $\alpha \nsubseteq \mu $. Since, $A$ satisfies \hyplA\ or \hypmuA\ and $B$ satisfies \hyplB\ or \hypmuB, we have, by Proposition \ref{charmincg}: $\beta ^{\perp B}\subseteq \mu $ and $\alpha ^{\perp B}=Cg_B(\alpha )^{\perp B}\subseteq \mu $.

Now let $\gamma \in {\cal K}(B)$ such that $\gamma \subseteq \beta ^{\perp B}$, arbitrary. Any $\varepsilon\in {\rm Min}(B)$ satisfies the following:

\noindent $\bullet$ if $Cg_B(\alpha )\subseteq \varepsilon $, then $[\gamma ,Cg_B(\alpha )]_B\subseteq \varepsilon $;

\noindent $\bullet$ if $Cg_B(\alpha )\nsubseteq \varepsilon $, then $\alpha \nsubseteq \varepsilon \cap \nabla _A$, that is $\Gamma (\varepsilon )=\varepsilon \cap \nabla_A\in D_A(\alpha )\cap {\rm Min}(A)=\Gamma (D_B(\beta)\cap {\rm Min}(B))$, hence $\varepsilon \in D_B(\beta)\cap {\rm Min}(B)$ since $\Gamma $ is a bijection, thus $\beta \nsubseteq \varepsilon $, so $\beta ^{\perp B}\subseteq \varepsilon $, again by Proposition \ref{charmincg}, thus $\gamma \subseteq \varepsilon $ by the above, hence $[\gamma ,Cg_B(\alpha )]_B\subseteq \varepsilon $.

Hence $[\gamma ,Cg_B(\alpha )]_B\subseteq \bigcap {\rm Min}(B)=\Delta _B$ since $B$ is semiprime, thus $[\gamma ,Cg_B(\alpha )]_B=\Delta _B$, that is $\gamma \subseteq Cg_B(\alpha )^{\perp B}=\alpha ^{\perp B}$. Therefore $\beta ^{\perp B}=\bigvee \{\gamma \in {\cal K}(B)\ |\ \gamma \subseteq \beta ^{\perp B}\}\subseteq \alpha ^{\perp B}$. Hence $A\subseteq B$ is an $r$--extension.\end{proof}

\begin{proposition} If $A\subseteq B$ is an admissible $r$--extension such that $A$ satisfies \hyplA\ or \hypmuA\ and $B$ satisfies \hyplB\ or \hypmuB , then the following are equivalent:\begin{enumerate}
\item\label{rigidext1} $A\subseteq B$ is a rigid extension;
\item\label{rigidext2} $\Gamma$ maps basic open sets of ${\cal M}in(B)$ to basic open sets of ${\cal M}in(A)$.\end{enumerate}\label{rigidext}\end{proposition}

\begin{proof} By Proposition \ref{rGammabij}, $A\subseteq B$ is an $m$--extension and $\Gamma$ is bijective.

\noindent (\ref{rigidext1})$\Rightarrow $(\ref{rigidext2}): Let $\beta \in {\rm PCon}(B)$, so that there exists $\alpha \in {\rm PCon}(A)$ with $\beta ^{\perp B}=\alpha ^{\perp B}=Cg_B(\alpha )^{\perp B}$ since $A\subseteq B$ is rigid.

By Proposition \ref{charmincg}, for any $\nu \in {\rm Min}(B)$, the following equivalences hold: $\beta \nsubseteq \nu $ iff $\beta ^{\perp B}\subseteq \nu $ iff $(Cg_B(\alpha ))^{\perp B}\subseteq \nu $ iff $Cg_B(\alpha )\nsubseteq \nu $ iff $\alpha \nsubseteq \nu $ iff $\alpha \nsubseteq \nu \cap \nabla _A$, therefore, since $\Gamma $ is bijective, we have $\Gamma(D_B(\beta)\cap {\rm Min}(B))=\{\nu \cap \nabla_A\ |\ \nu \in {\rm Min}(B),\beta \nsubseteq \nu \}=\{\mu \in {\rm Min}(A)\ |\ \alpha \nsubseteq \mu \}=D_A(\alpha)\cap {\rm Min}(A)$.

\noindent (\ref{rigidext2})$\Rightarrow $(\ref{rigidext1}): Let $\beta\in {\rm PCon}(B)$. By to the hypothesis of this implication, $\Gamma(D_B(\beta)\cap {\rm Min}(B))=D_A(\alpha )\cap {\rm Min}(A)$ for some $\alpha \in {\rm PCon}(A)$, thus $\{\nu \cap \nabla_A\ |\ \nu \in {\rm Min}(A),\beta \nsubseteq \nu \}=\{\mu \ |\ \mu \in {\rm Min}(A),\alpha \nsubseteq \mu \}$.

By Proposition \ref{charmincg}, it follows that any $\nu \in {\rm Min}(B)$ satisfies: $\beta ^{\perp B}\subseteq \nu $ iff $\beta \nsubseteq \nu $ iff $\alpha \nsubseteq \nu \cap \nabla _A$ iff $\alpha \nsubseteq \nu $  iff $Cg_B(\alpha )\nsubseteq \nu $ iff $(Cg_B(\alpha ))^{\perp B}\subseteq \nu $. As in the proof of the implication (\ref{GammahomStone3})$\Rightarrow $(\ref{GammahomStone1}) from Theorem \ref{GammahomStone}, it follows that $\beta ^{\perp B}=\bigcap V_B(\beta ^{\perp B})=\bigcap V_B(Cg_B(\alpha)^{\perp B})=Cg_B(\alpha )^{\perp B}=\alpha ^{\perp B}$. Hence the extension $A\subseteq B$ is rigid.\end{proof}

\begin{proposition} If $A\subseteq B$ is an admissible $r^*$--extension such that $A$ satisfies \hyplA\ or \hypmuA\ and $B$ satisfies \hyplB\ or \hypmuB , then the following are equivalent:\begin{enumerate}
\item\label{qrigidext1} $A\subseteq B$ is a quasirigid extension;
\item\label{qrigidext2} $\Gamma $ maps basic open sets of ${\cal M}in(B)^{-1}$ to basic open sets of ${\cal M}in(A)^{-1}$.
\end{enumerate}\label{qrigidext}\end{proposition}

\begin{proof} By Proposition \ref{r*Gammabij}, $A\subseteq B$ is an $m$--extension and $\Gamma $ is bijective.

\noindent (\ref{qrigidext1})$\Rightarrow $(\ref{qrigidext2}): Let $\beta\in {\cal K}(B)$, so that $\beta =\bigvee_{i=1}^n\beta _i$ for some $n\in \N ^*$ and some $\beta _1,\ldots ,\beta _n\in {\rm PCon}(B)$. By the hypothesis of this implication, for each $i\in \overline{1,n}$, there exists $\alpha _i\in {\cal K}(A)$ such that $\beta _i^{\perp B}=\alpha _i^{\perp B}=(Cg_B(\alpha _i))^{\perp B}$. 

Analogously to the proof of (\ref{rigidext1})$\Rightarrow $(\ref{rigidext2}) from Proposition \ref{rigidext}, it follows that $\Gamma (V_B(\beta _i)\cap {\rm Min}(B))=V_A(\alpha _i)\cap {\rm Min}(A)$ for all $i\in \overline{1,n}$, hence $\Gamma (V_B(\beta )\cap {\rm Min}(B))=\Gamma (\bigcap _{i=1}^nV_B(\beta _i)\cap {\rm Min}(B))=\bigcap _{i=1}^n\Gamma (V_B(\beta _i)\cap {\rm Min}(B))=\bigcap _{i=1}^n V_A(\alpha _i)\cap {\rm Min}(A)=V_A(\alpha )\cap {\rm Min}(A)$, where $\alpha =\bigvee_{i=1}^n\alpha _i\in {\cal K}(A)$.

\noindent (\ref{qrigidext2})$\Rightarrow $(\ref{qrigidext1}): Similar to the proof of (\ref{rigidext2})$\Rightarrow $(\ref{rigidext1}) in Proposition \ref{rigidext}.\end{proof}

\begin{corollary} If $A\subseteq B$ is an admissible $r$--extension and $r^*$--extension such that $A$ satisfies \hyplA\ or \hypmuA\ and $B$ satisfies \hyplB\ or \hypmuB , then the following are equivalent:\begin{itemize}
\item $A\subseteq B$ is a quasirigid extension;
\item $A\subseteq B$ is a rigid extension.
\end{itemize}\label{corrigidext}\end{corollary}

\begin{proof} By Proposition \ref{r*Gammabij}, $A\subseteq B$ is an $m$--extension and $\Gamma$ is bijective.

Clearly, if the extension $A\subseteq B$ is rigid, then it is quasirigid.

Now assume that $A\subseteq B$ is quasirigid. Then, by Proposition \ref{qrigidext}, $\Gamma $ maps basic open sets of ${\cal M}in(B)^{-1}$ to basic open sets of ${\cal M}in(A)^{-1}$. Since $\Gamma$ is bijective, it follows that $\Gamma $ maps basic open sets of ${\cal M}in(B)$ to basic open sets of ${\cal M}in(A)$, hence $A\subseteq B$ is rigid according to Proposition \ref{rigidext}.\end{proof}

\begin{theorem} If $A\subseteq B$ is an admissible extension such that $A$ satisfies \hyplA\ or \hypmuA\ and $B$ satisfies \hyplB\ or \hypmuB , then the following are equivalent:\begin{enumerate}
\item\label{Gammahomflat1} $A\subseteq B$ is a quasi $r^{\ast}$--extension;
\item\label{Gammahomflat2} $\Gamma: {\cal M}in(A)^{-1}\rightarrow {\cal M}in(B)^{-1}$ is a homeomorphism.\end{enumerate}\label{Gammahomflat}\end{theorem}

\begin{proof} By adapting the proof of Theorem \ref{GammahomStone}.\end{proof}

We say that $A$ satisfies the {\em annihilator condition} ({\rm AC} for short) if for all $\alpha,\beta\in {\rm PCon}(A)$ there exists $\gamma\in {\rm PCon}(A)$ such that $\gamma^{\perp A}=\alpha^{\perp A}\cap \beta^{\perp A}$.

\begin{remark} By Proposition \ref{negmin}.(\ref{negmin2}), $A$ satisfies AC if and only if the family $\{V_A(\alpha)\cap {\rm Min}(A)\ |\ \alpha \in {\rm PCon}(A)\}$ is closed under finite intersections. Thus, for any semiprime algebra $A$ that satisfies AC, the family $\{V_A(\alpha)\cap {\rm Min}(A)\ |\ \alpha \in {\rm PCon}(A)\}$ is a basis for the inverse topology ${\cal F}_{{\rm Min}(A)}$ of ${\rm Min}(A)$.\label{basisAC}\end{remark}

\begin{proposition} Let $A\subseteq B$ be an admissible extension such that $A$ satisfies \hyplA\ or \hypmuA\ and $B$ satisfies \hyplB\ or \hypmuB .\begin{enumerate}
\item\label{ACr*0} If $A$ satisfies AC, then: $A\subseteq B$ is a quasi $r^{\ast}$--extension iff $A\subseteq B$ is an $r^{\ast}$--extension.
\item\label{ACr*1} If $A\subseteq B$ is an $r$--extension and both $A$ and $B$ satisfy AC, then: $A\subseteq B$ is a quasirigid extension iff $A\subseteq B$ is a rigid extension.\end{enumerate}\label{ACr*}\end{proposition}

\begin{proof} (\ref{ACr*0}) Assume that $A\subseteq B$ is a quasi $r^{\ast}$--extension. Then, by Theorem \ref{Gammahomflat}, $\Gamma: {\cal M}in(A)^{-1}\rightarrow {\cal M}in(B)^{-1}$ is a homeomorphism.

Let $\nu \in {\rm Min}(B)$ and $\beta \in {\rm PCon}(B)$ such that $\beta \subseteq\nu $, so that $\nu \in V_B(\beta)\cap {\rm Min}(B)$, thus, by the above, $\nu \cap \nabla _A=\Gamma (\nu )\in \Gamma (V_B(\beta )\cap {\rm Min}(B))$, which, according to Remark \ref{basisAC}, equals $V_A(\alpha )\cap {\rm Min}(A)$ for some $\alpha \in {\rm PCon}(A)$. As in the proof of (\ref{rigidext2})$\Rightarrow $(\ref{rigidext1}) from Proposition \ref{rigidext}, it follows that $\alpha ^{\perp B}\subseteq \beta ^{\perp B}$. Therefore $A\subseteq B$ is an $r^{\ast}$--extension.

The converse implication is trivial.

\noindent (\ref{ACr*1}) By Propositions \ref{rigidext} and \ref{qrigidext} and the clear fact that, in this case, condition (\ref{rigidext2}) from Proposition \ref{rigidext} is equivalent to (\ref{qrigidext2}) from Proposition \ref{qrigidext}.\end{proof}

Let us denote, for any subset $X\subseteq A^2$, by $S(X)=\{\psi\in {\rm Min}(B)\ |\ X\nsubseteq \psi\cap \nabla_A\}$, thus $S(X)=\{\psi\in {\rm Min}(B)\ |\ \psi\cap \nabla_A\in D_A(Cg_A(X))\}$.

\begin{proposition} Let $A\subseteq B$ be an admissible extension with the property that, for any $\theta ,\zeta \in {\rm Con}(A)$, $\theta ^{\perp A}=\zeta ^{\perp A}$ implies $\theta ^{\perp B}=\zeta ^{\perp B}$. Assume that $A$ satisfies \hyplA\ or \hypmuA\ and that $B$ is hyperarchimedean and satisfies \hyplB\ or \hypmuB . Then the following are equivalent:\begin{enumerate}
\item\label{pbBaer1} ${\cal M}in(A)$ is a compact space;
\item\label{pbBaer2} $A\subseteq B$ is an $m$--extension;
\item\label{pbBaer3} For any $\alpha \in {\rm PCon}(A)$ there exists $\beta\in {\cal K}(A)$ such that $S(\beta )={\rm Spec}(B)\setminus S(\alpha )$;
\item\label{pbBaer4} For any $\alpha \in {\rm PCon}(A)$ there exists $\beta \in {\cal K}(A)$ such that $\beta \subseteq \alpha ^{\perp A}$ and $(\alpha \vee \beta )^{\perp A}=\Delta _A$.\end{enumerate}\label{pbBaer}\end{proposition}

\begin{proof} (\ref{pbBaer1})$\Leftrightarrow $(\ref{pbBaer4}) By Theorem \ref{charMincp}.

\noindent (\ref{pbBaer4})$\Rightarrow $(\ref{pbBaer3}): Assume that $\alpha\in {\rm PCon}(A)$. By the hypothesis (\ref{pbBaer4}), there exists $\beta\in {\cal K}(A)$ such that $\beta\subseteq \alpha^{\perp A}$ and $(\alpha\vee \beta)^{\perp A}=\Delta_A$.

In order to show that ${\rm Spec}(B)\setminus S(\alpha )=S(\beta )$, let $\psi\in {\rm Spec}(B)\setminus S(\alpha )$, hence $\alpha \subseteq \psi \cap\nabla _A$. Since $A\subseteq B$ satisfies the implication in the enunciation, by Proposition \ref{charmincg} it follows that these implications hold: if $(\alpha \vee \beta )^{\perp A}=\Delta_A $, that is $(Cg_B(\alpha \vee \beta))^{\perp A}=\Delta_A $, then $(Cg_B(\alpha \vee \beta ))^{\perp B}\subseteq \psi $, thus $Cg_B(\alpha \vee \beta )\nsubseteq \psi $, so $\alpha \vee \beta \nsubseteq \psi $, thus $\beta \nsubseteq \psi \cap \nabla _A$, so $\psi \in S(\beta)$, which proves  the inclusion ${\rm Spec}(B)\setminus S(\alpha)\subseteq S(\beta)$.

Conversely, let $\psi \in S(\beta )$, so that $\beta \nsubseteq \psi \cap \nabla _A$. But $[\alpha ,\beta ]_A=\Delta_A\subseteq \psi \cap \nabla _A\in {\rm Spec}(A)$, hence $\alpha \subseteq \psi \cap \nabla _A$, so $\psi \in {\rm Spec}(B)\setminus S(\alpha )$. Therefore $S(\beta)\subseteq {\rm Spec}(B)\setminus S(\alpha )$.

\noindent (\ref{pbBaer3})$\Rightarrow $(\ref{pbBaer2}): We have to prove that $\psi\cap \nabla _A\in {\rm Min}(A)$ for any $\psi \in {\rm Min}(B)$. Assume by absurdum that there exists $\psi \in {\rm Min}(B)$ such that $\psi \cap \nabla _A\notin {\rm Min}(A)$. But $\psi \cap \nabla _A\in {\rm Spec}(A)$ since $A\subseteq B$ is admissible, thus $\phi \subsetneq \psi \cap \nabla _A$ for some $\phi \in {\rm Min}(A)$. By Corollary \ref{Gammasurj}, there exists $\varepsilon \in {\rm Min}(B)$ such that $\phi =\varepsilon \cap \nabla _A$. So $\varepsilon \cap \nabla _A\subsetneq \psi \cap \nabla_ A$, hence there exists $(a,b)\in (\psi \cap \nabla_ A)\setminus (\varepsilon \cap \nabla _A)$, so that, if we denote by $\alpha =Cg_A(a,b)\in {\rm PCon}(A)$, then $\alpha \nsubseteq \varepsilon \cap \nabla_A$ and $\alpha \subseteq \psi \cap \nabla _A$, therefore $\varepsilon \in S(\alpha )$ and $\psi \notin S(\alpha )$. Since $S(\beta )={\rm Spec}(B)\setminus S(\alpha )$, it follows that $\varepsilon \notin S(\beta )$ and $\psi \in S(\beta )$, hence $\beta \subseteq\varepsilon \cap \nabla _A\subseteq \psi \cap \nabla_A$ and $\beta \nsubseteq \psi\cap \nabla _A$. We have obtained a contradiction, thus $A\subseteq B$ is an $m$--extension.

\noindent (\ref{pbBaer2})$\Rightarrow $(\ref{pbBaer1}): Assume that $A\subseteq B$ is an $m$--extension, so the map $\Gamma $ is surjective and continuous w.r.t. the Stone topologies by Corollary \ref{Gammasurj} and Proposition \ref{Gammacont}.(\ref{Gammacont0}).

By \cite[Theorem~$8$]{retic}, it follows that the reticulation ${\cal L}(B)$ of the hyperarchimedean algebra $B$ is a Boolean algebra. Since ${\cal M}in(B)$ and ${\cal M}in_{Id}({\cal L}(B))$ are homeomorphic, it follows that ${\cal M}in(B)$ is a Boolean space, hence ${\cal M}in(B)$ is a compact space, therefore ${\cal M}in(A)$ is also a compact space.\end{proof}

\begin{remark} Let $A$ be a reduced (that is semiprime) commutative ring and $Q(A)$ the complete ring of $A$ (see \cite{lam}). In this case, $Q(A)$ is a regular ring \cite{lam}, i.e. a hyperarchimedean ring. In accordance with \cite[Proposition~$7.2.(2)$]{picavet}, $A\subseteq Q(A)$ is a Baer extension of rings, so one can apply our Proposition 6.18. Then we obtain \cite[Theorem~$4.3$]{huckaba} as a particular case. It also results that, if $A$ is a reduced ring, then: ${\cal M}in(A)$ is compact iff $A\subseteq Q(A)$ is an $m$--extension.\end{remark}

\begin{theorem} If $A\subseteq B$ is an admissible $m$--extension such that \hypmuA , \hypcpA ,\hypmuB\ and \hypcpB\ are satisfied and $\Gamma $ is injective, then $\Gamma :{\cal M}in(B)\rightarrow {\cal M}in(A)$ is a homeomorphism and $A\subseteq B$ is a weak rigid extension.\label{thGammacont}\end{theorem}

\begin{proof} We will be using Propositions \ref{charmincg} and \ref{intersperp} and Lemma \ref{proprperp}.(\ref{proprperp1}).

Let $\alpha \in {\rm Con}(A)$ and $\mu \in {\rm Min}(A)$, so that $\mu =\nu \cap \nabla _A$ for some $\nu \in {\rm Min}(B)$ by Corollary \ref{Gammasurj}. Then:

$\mu \in D_A(\alpha )\cap {\rm Min}(A)$, that is $\alpha \nsubseteq \mu $, is equivalent to $\alpha ^{\perp A}\subseteq \mu $, which implies $\alpha ^{\perp B}=Cg
_B(\alpha )^{\perp B}=Cg_B(\alpha ^{\perp A})\subseteq Cg_B(\mu )\subseteq \nu $, thus $Cg
_B(\alpha )\nsubseteq \nu $, that is $\nu \in D_B(Cg_B(\alpha ))\cap {\rm Min}(B)$;

on the other hand, $\nu \in D_B(Cg_B(\alpha ))\cap {\rm Min}(B)$ means that $Cg_B(\alpha )\nsubseteq \nu $, so that $\alpha ^{\perp B}=Cg
_B(\alpha )^{\perp B}\subseteq \nu $, hence $\alpha ^{\perp A}=\alpha ^{\perp B}\cap \nabla _A\subseteq \nu \cap \nabla _A=\mu $, thus $\alpha \nsubseteq \mu $, that is $\mu \in D_A(\alpha )\cap {\rm Min}(A)$;

hence $\nu \in D_B(Cg_B(\alpha ))\cap {\rm Min}(B)$ iff $\Gamma(\nu )=\mu \in D_A(\alpha )\cap {\rm Min}(A)$ iff $\nu \in \Gamma ^{-1}(D_A(\alpha )\cap {\rm Min}(A))$, therefore $\Gamma ^{-1}(D_A(\alpha )\cap {\rm Min}(A))=D_B(Cg_B(\alpha ))\cap {\rm Min}(B)$.

Thus $\Gamma $ is continuous and, by Proposition \ref{negmin}, (\ref{negmin1}), for all $\theta \in {\rm Con}(A)$, $\Gamma ^{-1}(V_A(\alpha )\cap {\rm Min}(A))=\Gamma ^{-1}(D_A(\alpha ^{\perp A})\cap {\rm Min}(A))=D_B(Cg_B(\alpha ^{\perp A}))\cap {\rm Min}(B)=V_B(Cg_B(\alpha ^{\perp A})^{\perp B})\cap {\rm Min}(B)=V_B(\alpha ^{\perp B\perp B})\cap {\rm Min}(B)=V_B(Cg_B(\alpha )^{\perp B\perp B})\cap {\rm Min}(B)=V_B(Cg_B(\alpha ))\cap {\rm Min}(B)$.

Hence, if $\Gamma $ is injective and thus bijective according to Corollary \ref{Gammasurj}, then $\Gamma $ is a homeomorphism, in particular $\Gamma $ is open, thus, for every $\beta \in {\rm Con}(B)$, there exists an $\alpha \in {\rm Con}(A)$ such that $\Gamma (D_B(\beta )\cap {\rm Min}(B))=D_A(\alpha )\cap {\rm Min}(A)$, hence, by the above, along with Proposition \ref{negmin}, (\ref{negmin1}), and Proposition \ref{intersperp}, $V_B(\beta ^{\perp B})\cap {\rm Min}(B)=D_B(\beta )\cap {\rm Min}(B)=\Gamma ^{-1}(D_A(\alpha )\cap {\rm Min}(A))=\Gamma ^{-1}(D_A(\alpha ^{\perp A\perp A})\cap {\rm Min}(A))=D_B(Cg_B(\alpha ^{\perp A\perp A}))\cap {\rm Min}(B)=D_B(Cg_B(\alpha ^{\perp A})^{\perp B}))\cap {\rm Min}(B)=V_B(Cg_B(\alpha ^{\perp A})))\cap {\rm Min}(B)=V_B(Cg_B(\alpha )^{\perp B}))\cap {\rm Min}(B)$, hence, according to Lemma \ref{negradthus}, $\beta ^{\perp B}=\bigcap (V_B(\beta ^{\perp B})\cap {\rm Min}(B))=\bigcap (V_B(Cg_B(\alpha )^{\perp B}))\cap {\rm Min}(B))=Cg_B(\alpha )^{\perp B}=\alpha ^{\perp B}$, thus $A\subseteq B$ is weak rigid.\end{proof}

\begin{theorem} If $A\subseteq B$ is an admissible weak $r$--extension such that \hypmuA , \hypcpA ,\hypmuB\ and \hypcpB\ are satisfied, then $A\subseteq B$ is an $m$--extension and an $r$--extension and $\Gamma :{\cal M}in(B)\rightarrow {\cal M}in(A)$ is a homeomorphism.\label{weakr}\end{theorem}

\begin{proof} Assume that $A\subseteq B$ is an admissible weak $r$--extension, so that $\Gamma $ is surjective by Corollary \ref{Gammasurj}. We will apply Proposition \ref{charmincg}.

Assume by absurdum that there exists a $\mu \in {\rm Min}(B)$ such that $\mu \cap \nabla _A\notin {\rm Min}(A)$, so that $\phi \subsetneq \mu \cap \nabla _A$ for some $\phi \in {\rm Min}(A)$. By the above, $\phi =\varepsilon \cap \nabla _A$ for some $\varepsilon \in {\rm Min}(B)$, so that $\varepsilon \cap \nabla _A\subsetneq \mu \cap \nabla _A$, thus $\varepsilon \neq \mu $, hence $\varepsilon \setminus \mu \neq \emptyset $, that is $(x,y)\in \varepsilon \setminus \mu $ for some $x,y\in B$.

We have $Cg_B(x,y)\nsubseteq \mu $, hence there exists an $\alpha \in {\rm Con}(A)$ such that $\alpha \nsubseteq \mu $ and $Cg_B(x,y)^{\perp B}\subseteq \alpha ^{\perp B}=Cg_B(\alpha )^{\perp B}$ since $A\subseteq B$ is a weak $r$--extension. But $\alpha \nsubseteq \mu $ implies $Cg_B(\alpha )\nsubseteq \mu $, hence $Cg_B(\alpha )^{\perp B}\subseteq \mu $.

Since $Cg_B(x,y)\subseteq \varepsilon $, we have $Cg_B(\alpha )^{\perp B}\supseteq Cg_B(x,y)^{\perp B}\nsubseteq \varepsilon $, thus $\alpha \subseteq Cg_B(\alpha )\subseteq \varepsilon $, hence $\alpha \subseteq \varepsilon \cap \nabla _A=\phi \subset \mu \cap \nabla _A\subseteq \mu $, hence $Cg_B(\alpha )\subseteq \mu $, thus $Cg_B(\alpha )^{\perp B}\nsubseteq \mu $, contradicting the above.

Therefore $A\subseteq B$ is an $m$--extension.

Now let $\mu ,\nu \in {\rm Min}(B)$ such that $\mu \cap \nabla _A=\nu \cap \nabla _A$. Assume by absurdum that $\mu \neq \nu $, so that $\mu \setminus \nu \neq \emptyset $, that is $(u,v)\in \mu \setminus \nu $ for some $u,v\in B$.

Then $Cg_B(u,v)\nsubseteq \nu $, thus, since $A\subseteq B$ is a weak $r$--extension, there exists a $\xi \in {\rm Con}(A)$ such that $\xi \nsubseteq \nu $ and $Cg_B(u,v)^{\perp B}\subseteq \xi ^{\perp B}$.

Since $Cg_B(u,v)\subseteq \mu $, $Cg_B(u,v)^{\perp B}\nsubseteq \mu $, hence, by Proposition \ref{intersperp}, $Cg_B(\xi ^{\perp A})=Cg_B(\xi )^{\perp B}=\xi ^{\perp B}\nsubseteq \mu $, thus $\xi \subseteq \xi ^{\perp A\perp A}\subseteq Cg_B(\xi ^{\perp A\perp A})=Cg_B(\xi ^{\perp A})^{\perp B}\subseteq \mu $, hence $\xi \subseteq \mu \cap \nabla _A=\nu \cap \nabla _A\subseteq \nu $, contradicting the above.

Therefore $\Gamma $ is injective and thus a homeomorphism by Theorem \ref{thGammacont}.

Finally, let $\mu \in {\rm Min}(B)$ and $\beta \in {\rm PCon}(B)$ such that $\beta \nsubseteq \mu $, so that, since $A\subseteq B$ is a weak $r$--extension, there exists a $\gamma \in {\rm Con}(A)$ such that $\gamma \nsubseteq \mu $ and $\beta ^{\perp B}\subseteq \gamma ^{\perp B}$.

Then $(w,z)\in \gamma \setminus \mu $ for some $w,z\in A$, so that $Cg_A(w,z)\nsubseteq \mu $ and $Cg_A(w,z)\subseteq \gamma $, so that $Cg_B(Cg_A(w,z))\subseteq Cg_B(\gamma )$ and thus $\beta ^{\perp B}\subseteq \gamma ^{\perp B}=Cg_B(\gamma )^{\perp B}\subseteq Cg_B(Cg_A(w,z))^{\perp B}=Cg_A(w,z)^{\perp B}$.

Therefore $A\subseteq B$ is an $r$--extension.\end{proof}

\begin{theorem} If $A\subseteq B$ is an admissible weak $r^*$--extension such that \hypmuA , \hypcpA ,\hypmuB\ and \hypcpB\ are satisfied, then $A\subseteq B$ is an $m$--extension and an $r^*$--extension and $\Gamma $ is a homeomorphism w.r.t. the Stone topologies.\label{weakrstar}\end{theorem}

\begin{proof} By adapting the proof of Theorem \ref{weakr}.\end{proof}

\begin{corollary} If \hypmuA , \hypcpA ,\hypmuB\ and \hypcpB\ are satisfied and $A\subseteq B$ is admissible and either a weak $r$--extension or a weak $r^*$--extension, then $A\subseteq B$ is a weak rigid extension.\end{corollary}

\begin{corollary} If $A\subseteq B$ is admissible and \hypmuA , \hypcpA ,\hypmuB\ and \hypcpB\ are satisfied, then the following are equivalent:\begin{itemize}
\item $A\subseteq B$ is a weak rigid extension;
\item $A\subseteq B$ is a weak $r$--extension;
\item $A\subseteq B$ is a weak $r^*$--extension;
\item $A\subseteq B$ is an $r$--extension;
\item $A\subseteq B$ is an $r^*$--extension.\end{itemize}\end{corollary}

\section*{Acknowledgements}

This work was supported by the research grant number IZSEZO\_186586/1, awarded to the project {\em Re\-ti\-cu\-la\-tions of Concept Algebras} by the Swiss National Science Foundation, within the programme Scientific Exchanges.

\end{document}